\documentclass[11pt]{article}

\usepackage[T1]{fontenc}
\usepackage{lmodern}

\usepackage{amsmath, amssymb, amsthm}
\usepackage{mathrsfs}
\usepackage{stmaryrd}
\usepackage[margin=2.5cm]{geometry}
\usepackage{color, graphicx, float}
\usepackage{booktabs}
\usepackage{subfigure}
\usepackage[justification=centering]{caption}
\usepackage{tikz}
\usepackage{multirow}
\usepackage{rotating}
\usepackage[ruled]{algorithm2e}
\usepackage{epstopdf}
\usepackage{etoolbox}

\makeatletter
\patchcmd{\@maketitle}{\LARGE \@title}{\LARGE\bfseries\@title}{}{}

\renewcommand{\@seccntformat}[1]{\csname the#1\endcsname.\quad}
\makeatother

\definecolor{darkblue}{rgb}{0,0,.5}

\usepackage{hyperref}
\hypersetup{
	colorlinks=true,		
	linkcolor=darkblue,		
	citecolor=darkblue,		
	urlcolor=darkblue 		
}

\makeatletter
\def\th@plain{%
	\thm@notefont{}
	\itshape 
}
\def\th@definition{%
	\thm@notefont{}
	\normalfont 
}

\renewenvironment{proof}[1][\proofname]{\par
	\normalfont
	\topsep0\p@\@plus3\p@ \trivlist
	\item[\hskip\labelsep\itshape
	#1\@addpunct{.}]\ignorespaces
}{%
	\qed\endtrivlist
}
\makeatother

\newtheorem{theorem}{Theorem}[section]
\newtheorem{lemma}{Lemma}[section]

\theoremstyle{definition}
\newtheorem{definition}{Definition}[section]
\newtheorem{example}{Example}[section]
\newtheorem{remark}{Remark}[section]
\newtheorem{assumption}{Assumption}[section]

\usepackage[shortlabels]{enumitem}

\allowdisplaybreaks

\newcommand{\argmin}{\ensuremath{\operatorname*{argmin}}}
\newcommand{\dom}{\ensuremath{\operatorname{dom}}}

\newcommand{\prox}{\ensuremath{\operatorname{Prox}}}
\newcommand{\dist}{\ensuremath{\operatorname{dist}}}

\begin{document}

\title{Variable Smoothing Alternating Proximal Gradient Algorithm for Coupled Composite Optimization\thanks{This work was supported by the National Natural Science Foundation of China (12271067), the NSF of Chongqing (CSTB2024NSCQ-MSX1282), the Education Committee Project Research Foundation of Chongqing (KJZD-K202500805), the Team Building Project for Graduate Tutors in Chongqing (yds223010), the Project of Chongqing Technology and Business University (yjscxx2025-269-238), and the Australian Research Council (ARC) Discovery Project DP230101749.}}

\author{Xian-Jun Long\thanks{\baselineskip 9pt School of Mathematics and Statistics, Chongqing Technology and Business University, Chongqing 400067, P.R.China. And Chongqing Key Laboratory of Statistical Intelligent Computing and Monitoring, Chongqing Technology and Business University, Chongqing, 400067, P.R.China. Email: xianjunlong@ctbu.edu.cn},
~Kang Zeng\thanks{\baselineskip 9pt School of Mathematics and Statistics, Chongqing Technology and Business University, Chongqing 400067, P.R.China. Email: zengkang000111@163.com},
~Gao-Xi Li\thanks{ School of Mathematics and Statistics, Chongqing Technology and Business University, Chongqing 400067, P.R.China. Email: ligaoxicn@126.com},
~Minh N. Dao\thanks{School of Science, RMIT University, Melbourne, VIC 3000, Australia. Email: minh.dao@rmit.edu.au},
~and Zai-Yun Peng\thanks{School of Mathematics, Yunnan Normal University, Kunming, 650092, P.R. China. Email: pengzaiyun@126.com} }

\date{October 31, 2025}

\maketitle

\begin{abstract}
In this paper, we consider a broad class of nonconvex and nonsmooth optimization problems, where one objective component is a nonsmooth weakly convex function composed with a linear operator. By integrating variable smoothing techniques with first-order methods, we propose a variable smoothing alternating proximal gradient algorithm that features flexible parameter choices for step sizes and smoothing levels. Under mild assumptions, we establish that the iteration complexity to reach an $\varepsilon$-approximate stationary point is $\mathcal{O}(\varepsilon^{-3})$. The proposed algorithm is evaluated on sparse signal recovery and image denoising problems. Numerical experiments demonstrate its effectiveness and superiority over existing algorithms.
\end{abstract}



\section{Introduction}
\label{sec:Introduction}

Consider the nonconvex nonsmooth composite optimization problem
\begin{equation}\label{eq:L1E1}
\min_{(x,y)\in \mathbb{R}^{n}\times\mathbb{R}^{m}} \mathcal{L}(x,y)=f(x) + g(Ay)+H(x,y),
\end{equation}
where $f:\mathbb{R}^{n}\rightarrow \mathbb{R}$ is a (possibly nonsmooth) convex function, $g:\mathbb{R}^{d}\rightarrow \mathbb{R}\cup \{+\infty\}$ is a nonsmooth and proper lower semicontinuous $\rho$-weakly convex function, $A:\mathbb{R}^{m} \rightarrow \mathbb{R}^{d}$ is a linear operator, and $H:\mathbb{R}^{n}\times \mathbb{R}^{m} \rightarrow \mathbb{R}$ is a continuously differentiable (possibly nonconvex) function. Such problem has a wide range of applications, including compressed sensing, machine learning, nonnegative matrix factorization, image denoising, signal recovery, and multimodal learning for image classification; see, for example \cite{BST,CP,CP1,DD,OBP,PN,TSR}.

When $A$ is the identity operator, model \eqref{eq:L1E1} has been discussed by many scholars in the literature; see, e.g., \cite{AB,BST,CL,GCH,PS,WH,ZD}. In particular, Bolte et al. \cite{BST} proposed the proximal alternating linearized minimization (PALM) algorithm
\begin{equation*}
\begin{aligned}
\left\{
\begin{array}{ll}
x_{k+1}\in \arg\min_{x\in \mathbb{R}^{n}} \{f(x)+\langle \nabla_{x}H(x_{k},y_{k}),x-x_{k}\rangle+\frac{c_{k}}{2}\|x-x_{k}\|^{2}\},\\
y_{k+1}\in \arg\min_{y\in \mathbb{R}^{m}} \{g(y)+\langle \nabla_{y}H(x_{k+1},y_{k}),y-y_{k}\rangle+\frac{d_{k}}{2}\|y-y_{k}\|^{2}\},
\end{array}
\right.
\end{aligned}
\end{equation*}
where $c_{k}>0$ and $d_{k}>0$. The global convergence result was proved using the Kurdyka--{\L}ojasiewicz property. Following this algorithm, Pock and Sabach presented the inertial version of PALM (iPALM) in \cite{PS}, Gao et al. introduced the Gauss--Seidel type inertial PALM algorithm (GiPALM) in \cite{GCH}, Wang and Han also presented a generalized inertial proximal alternating linearized minimization algorithm in \cite{WH}. These methods significantly improve computational efficiency.

It is worth mentioning that PALM, iPALM, and GiPALM require two evaluations of the proximal operator for nonconvex and nonsmooth functions. However, the proximal operator is generally difficult to calculate for a nonconvex and nonsmooth function.

On the other hand, smooth approximations for optimization problems have been extensively studied in recent years because they convert nonsmooth problems into smooth ones, thereby enabling efficient solutions via gradient descent methods. Bo\c{t} et al. \cite{BB,BC} presented variable smoothing algorithms for convex optimization problems. Bohm and Wright \cite{BW} later extended these results to the weakly convex case. Recently, Liu and Xia \cite{LX} proposed a proximal variable smoothing gradient algorithm for a nonconvex and nonsmooth minimization problem, which is a special case of problem \eqref{eq:L1E1}. They established an $\mathcal{O}(\varepsilon^{-3})$ complexity to achieve an $\varepsilon$-approximate solution.

Motivated by the works of \cite{GCH,BW,LX}, we propose a variable smoothing alternating proximal gradient algorithm to solve problem \eqref{eq:L1E1}. We construct a partially variable smoothed approximation of the objective by using a smooth approximation of $g$, known as the Moreau envelope and denoted by $g_{\mu}$,
\begin{align}\label{guanghua}
\min_{(x,y)\in \mathbb{R}^{n}\times\mathbb{R}^{m}} \mathcal{L}_\mu(x,y)=f(x)+g_{\mu}(Ay)+H(x,y).
\end{align}
This approximation explicitly separates the smooth components $(g_{\mu},H)$ from the nonsmooth term $f$, allowing efficient optimization via standard first-order methods with flexible choices of step sizes and smoothing parameters. Under appropriate assumptions, we establish a complexity bound of $O(\varepsilon^{-3})$ to find an $\varepsilon$-approximate stationary point.

The rest of this paper is organized as follows. Section~\ref{sec:Preliminaries} introduces key concepts and preliminary results. In Section~\ref{sec:Stationary} discusses approximate stationary points and the standing assumptions. Section~\ref{sec:Algorithm} presents the variable smoothing alternating proximal gradient algorithm and analyzes its convergence properties. Finally, numerical experiments in Section~\ref{sec:Numerical} demonstrate the effectiveness of the proposed algorithm.

\section{Preliminaries}
\label{sec:Preliminaries}

Let $\mathbb{R}^n$ be a finite-dimensional Euclidean space, which equipped with standard inner product $\langle \cdot,\cdot \rangle$ and norm $\|\cdot \|$, respectively. For any $x,y\in \mathbb{R}^n$, $\|(x,y)\|:=\sqrt{\|x\|^2+\|y\|^2}$. Given a nonempty set $C\subseteq \mathbb{R}^n$, the distance from $z\in \mathbb{R}^n$, the \emph{distance} from $z$ to $C$ is defined as
$\dist(z,C):=\inf_{w\in C}\|w-z\|$. Let $h: \mathbb{R}^n\rightarrow {\mathbb{R}}\cup \{+\infty\}$ be a mapping, the \emph{domain} of $h$ is defined by $\dom h:=\{x\in \mathbb{R}^n:h(x)<+\infty\}$. The function $h$ is said to be \emph{proper} if $\dom h \neq\varnothing$ and \emph{lower semicontinuous} if $h(x)\leq \liminf_{z\rightarrow x}h(z)$ for any $x\in \mathbb{R}^n$.

Let $F:\mathbb{R}^n\rightarrow \mathbb{R}\cup \{+\infty\}$ be a proper lower semicontinuous convex function and $\lambda>0$. The \emph{proximal operator} of $F$ at $v\in \mathbb{R}^n$ is defined by
\begin{align*}
\prox_{\lambda F}(v):=\argmin_{x\in \mathbb{R}^n} \left(F(x)+\frac{1}{2\lambda}\|x-v\|^2\right).
\end{align*}
It is known that $\prox_{\lambda F}$ is nonexpansive, i.e., for all $x,y\in \mathbb{R}^n$,
\begin{align*}
\|\prox_{\lambda F}(x)-\prox_{\lambda F}(y)\|\leq\|x-y\|.
\end{align*}
The \emph{regular subdifferential} of $F$ at $x\in\dom F$ is defined by
$$\widehat{\partial}F(x):=\left\{u\in\mathbb{R}^{n}:\liminf_{y\rightarrow x\atop y\neq x} \frac{\ F(y)- F(x) + \langle u, y-x \rangle}{\|y-x\|}\geq0\right\}.$$
The \emph{limiting subdifferential} of $F$ at $x\in\dom F$ is defined by
 \begin{align*}
\partial F(x) :=\{u \in\mathbb{R}^{n}: \exists\ x_{k}\rightarrow x, \widehat{\partial}F(x_{k}) \ni u_{k} \rightarrow u \ \mathrm{with} \  F(x_{k}) \rightarrow F(x)\}.
\end{align*}
When $F$ is a convex function, both subdifferentials coincide with the subdifferential in the sense of convex analysis, i.e.,
\begin{align*}
\widehat{\partial}F(x)=\partial F(x)=\{u\in\mathbb{R}^{n}: \forall y\in\mathbb{R}^{n},\quad F(y)-F(x)\geq\langle{u,y-x}\rangle\}.
\end{align*}
Clearly, $\widehat{\partial}F(x) \subseteq \partial F(x)$ for all $x\in \mathbb{R}^{n}$ and both of them are closed.  If $G:\mathbb{R}^n\rightarrow \mathbb{R}$ is  a continuously differentiable function, then $\partial G(x)=\{\nabla G(x)\}$ and $\partial(F+G)(x)=\partial F(x)+\nabla G(x)$ for all $x\in \mathbb{R}^n$, where $\nabla G(x)$ denotes the gradient of $G$ at $x$.

\begin{lemma}[{\cite[Theorem 2.64]{BC}}]
\label{l:descent}
Let $F:\mathbb{R}^n\rightarrow \mathbb{R}$ be a continuously differentiable function whose gradient $\nabla F$ is $L$-Lipschitz continuous with $L>0$. Then, for all $x,y\in \mathbb{R}^n$,
\begin{align*}
|F(y)-F(x)-\langle\nabla F(x),y-x\rangle|\leq\frac{L}{2}\|y-x\|^2.
\end{align*}
\end{lemma}

\begin{definition}[\cite{V}] A function $F: \mathbb{R}^n\rightarrow (-\infty,+\infty]$  is said to be \emph{$\rho$-weakly convex} if $F +\frac{\rho}{2}\|\cdot\|^2$ is convex.
\end{definition}

\begin{remark}
Obviously, a smooth function having a Lipschitz gradient is weakly convex.
\end{remark}

\begin{definition}[{\cite[Definition 2.1]{BW}}] Let $F: \mathbb{R}^n\rightarrow (-\infty,+\infty]$ be a proper $\rho$-weakly convex lower semicontinuous function. The Moreau envelope function of $F$ is defined as
\begin{align*}
F_{\mu}(x):=\min_{y\in \mathbb{R}^n}\{F(y)+\frac{1}{2\mu}\|y-x\|^2\},
\end{align*}
where $\mu \in(0,1/\rho)$.
\end{definition}

\begin{lemma}
\label{l:weakcvx}
Let $F: \mathbb{R}^n\rightarrow (-\infty,+\infty]$ be a proper $\rho$-weakly convex and lower semicontinuous function, and let $\mu \in(0,1/\rho)$. Then
\begin{enumerate}
\item
\cite[Corollary 3.4]{HLO} The Moreau envelope function $F_{\mu}$ is continuously differentiable on $\mathbb{R}^n$ and, for all $x\in \mathbb{R}^n$,
\begin{align*}
\nabla F_{\mu}(x) =\frac{1}{\mu}(x-\prox_{\mu F}(x)).
\end{align*}
This gradient is $\max\{\frac{1}{\mu},\frac{\rho}{1-\rho \mu}\}$-Lipschitz continuous.
\item\label{l:weakcvx_inclu}
\cite[Lemma 3.2]{BW} $\nabla{F_{\mu}(x)}\in \partial F(\prox_{\mu F}(x))$.
\end{enumerate}
\end{lemma}

\begin{lemma}[{\cite[Lemma 3.3]{BW}}]\label{2.6} Let $F: \mathbb{R}^n\rightarrow (-\infty,\infty)$ be a $\rho$-weakly convex function and $L_F$-Lipschitz continuous, and let $\mu \in (0,1/\rho)$. Then the Moreau envelope $F_\mu$ is Lipschitz continuous and, for all $x\in \mathbb{R}^n$,
\begin{align*}
\|\nabla F_\mu(x)\|\leq L_F \text{~~and~~} \|x-\prox_{\mu F}(x)\|\leq \mu L_F.
\end{align*}
\end{lemma}

\begin{lemma}[{\cite[Lemma~\ref{l:LL}]{BW}}]
\label{l:MoreauEnv}
Let $F: \mathbb{R}^n\rightarrow (-\infty,+\infty]$ be a proper $\rho$-weakly convex lower semicontinuous function, let $\mu_1$ and $\mu_2$ be parameters such that $0< \mu_2 \leq \mu_1 < 1/\rho$. Then
\begin{align*}
F_{\mu_2}(x)\leq F_{\mu_1}(x)+\frac{\mu_1(\mu_1-\mu_2)}{2\mu_2}\|
\nabla F_{\mu_1}(x)\|^2.
\end{align*}
If $F$ is additionally $L_F$-Lipschitz continuous, then
\begin{align*}
F_{\mu_2}\leq F_{\mu_1}+\frac{\mu_1(\mu_1-\mu_2)}{2\mu_2}(L_F)^2.
\end{align*}
\end{lemma}

We end this section by the following technical lemma.
\begin{lemma}\label{l:estimate}
Let $x\in [1, +\infty)$ and $\alpha\in [0, 1]$. Then
\begin{align*}
(1 +x)^{\alpha} - 1 \geq \alpha (\ln 2) x^{\alpha}.
\end{align*}
\end{lemma}
\begin{proof}
Let $\varphi(x) :=(1 +x)^{\alpha} -1 - \alpha (\ln 2) x^{\alpha}$. Then $\psi(\alpha) =:\varphi(1) =2^{\alpha} -1 -\alpha \ln 2$. We see that, for all $\alpha \geq 0$,
$\psi'(\alpha) =(\ln 2) (2^{\alpha} -1) \geq 0$. It follows that $\varphi(1) =\psi(\alpha) \geq \psi(0) =0$.

On the other hand,
\begin{align*}
\varphi'(x) =\alpha(1 +x)^{\alpha -1} -\alpha^2 (\ln 2) x^{\alpha -1} =\alpha x^{\alpha -1}\left(\left(1 +\frac{1}{x}\right)^{\alpha -1} -\alpha \ln 2\right).
\end{align*}
As $x \geq 1$ and $\alpha -1 \leq 0$, it holds that
\begin{align*}
\left(1 +\frac{1}{x}\right)^{\alpha -1} -\alpha \ln 2 &\geq (1 + 1)^{\alpha -1} -\alpha \ln 2 =2^{\alpha} -2^{\alpha -1} -\alpha \ln 2 \\
&\geq 2^{\alpha} -1 -\alpha \ln 2 =\varphi(1) \geq 0.
\end{align*}
Therefore, for all $x\geq 1$, $\varphi'(x) \geq 0$, which implies that $\varphi(x) \geq \varphi(1) =0$. The proof is complete.
\end{proof}

\section{Approximate stationary points}
\label{sec:Stationary}

We begin this section by providing an equivalent condition for $(x, y)$ to be a stationary point of problem \eqref{eq:L1E1}. This characterization will be instrumental in formulating our approximate stationarity measure.

Define
\begin{align*}
G_\lambda(x,y) =\frac{1}{\lambda}(x -\prox_{\lambda f}(x -\lambda\nabla_x  H(x,y))).
\end{align*}
We have the following characterization.
\begin{lemma}
\label{l:stationary}
Let $(x,y)\in \mathbb{R}^n\times \mathbb{R}^m$ and $\lambda\in (0, +\infty)$. Then $0\in \partial\mathcal{L}(x,y)$ if and only if
\begin{align*}
\dist(0,G_\lambda(x,y)) +\dist(-\nabla_y H(x,y), A^*\partial g(Ay)) =0
\end{align*}
\end{lemma}
\begin{proof}
We first have that
\begin{align}\label{zhu2.1}
\partial\mathcal{L}(x,y)&=(\partial_x\mathcal{L}(x,y),\partial_y\mathcal{L}(x,y))\notag\\
&=(\partial f(x)+
\nabla_xH(x,y),A^*\partial g(Ay)+
\nabla_yH(x,y)).
\end{align}
Next, observe that
\begin{align*}
0\in \partial f(x)+
\nabla_x H(x,y)
&\iff 0\in \partial f(x) +\frac{1}{\lambda}(x -x +\lambda\nabla_x H(x,y)) \\
&\iff x = \argmin_{z\in \mathbb{R}^n} f(z) +\frac{1}{2\lambda}\|z -x +\lambda\nabla_x H(x,y)\|^2 \\
&\iff x = \prox_{\lambda f}(x -\lambda\nabla_x H(x,y)).
\end{align*}
Therefore,
\begin{align*}
0\in \partial\mathcal{L}(x,y)
&\iff
\begin{cases}
0\in \partial f(x)+
\nabla_x H(x,y), \\
0\in A^*\partial g(Ay)+
\nabla_y H(x,y)
\end{cases} \\
&\iff
\begin{cases}
x = \prox_{\lambda f}(x -\lambda\nabla_x H(x,y)), \\
-\nabla_y H(x,y) \in A^*\partial g(Ay).
\end{cases} \\
&\iff
\begin{cases}
\dist(0,G_\lambda(x,y)) =0, \\
\dist(-\nabla_y H(x,y), A^*\partial g(Ay)) =0,
\end{cases}
\end{align*}
which completes the proof.
\end{proof}

To establish the complexity bound of our proposed algorithms, we introduce a new convergence measure and define an $\varepsilon$-approximate stationary point of problem \eqref{eq:L1E1}. Traditionally, convex problems adopt the optimality gap $\mathcal{L}(x)-\mathcal{L}(x^*)$ as the convergence criterion (see, e.g., \cite{TSR}), while nonconvex and nonsmooth problems utilize the gradient mapping (see \cite{BW,GLZ,LX}). In view of Lemma~\ref{l:stationary}, we introduce the following definition of approximate stationarity.

\begin{definition}
Let $(x^*,y^*)\in \mathbb{R}^n\times \mathbb{R}^m$. We say that $(x^*,y^*)$ is an \emph{$\varepsilon$-approximate stationary point} of problem~\eqref{eq:L1E1} if
\begin{align}\label{wending}
\dist(0,G_\lambda(x^*,y^*))+\dist(-\nabla_yH(x^*,y^*),A^*\partial g(Ay^*))\leq\varepsilon.
\end{align}
Similarly, $(x^*,y^*)$ is said to be an \emph{$\varepsilon$-approximate stationary point} of problem~\eqref{guanghua} if
\begin{align}\label{wending1}
\dist(0,G_\lambda(x^*,y^*))+\dist(-\nabla_yH(x^*,y^*),A^*\nabla g_\mu(Ay^*))\leq\varepsilon.
\end{align}
\end{definition}

For problem~\eqref{eq:L1E1}, the following assumptions are to be considered.

\begin{assumption}
\label{a:H}
For any $y\in \mathbb{R}^m$, the partial gradient $\nabla_x H(\cdot,y)$ is $L_{11}$-Lipschitz continuous, i.e.,
\begin{align*}
\forall x_1,x_2\in \mathbb{R}^n,\quad \|\nabla_x H(x_1,y)-\nabla_x H(x_2,y)\|\leq L_{11}\|x_1-x_2\|.
\end{align*}
For any $x\in \mathbb{R}^n$, the partial gradient $\nabla_x H(x,\cdot)$ is $L_{12}$-Lipschitz continuous and the partial gradient $\nabla_y H(x,\cdot)$ is $L_{22}$-Lipschitz continuous, i.e.,
\begin{align*}
&\forall y_1,y_2\in \mathbb{R}^n,\quad \|\nabla_x H(x,y_1)-\nabla_x H(x,y_2)\|\leq L_{12}\|y_1-y_2\|,\\
&\forall y_1,y_2\in\mathbb{ R}^m,\quad \|\nabla_y H(x,y_1)-\nabla_y H(x,y_2)\|\leq L_{22}\|y_1-y_2\|.
\end{align*}
\end{assumption}

\begin{assumption}
\label{a:g}
The function $g$ is $\rho$-weakly convex and $L_g$-Lipschitz continuous. The function $\mathcal{L}$ is bounded below.
\end{assumption}

\begin{remark}
\label{r:G(x,.)}
From the Lipschitz continuity of $\nabla_x H(x,\cdot)$ in Assumption~\ref{a:H}, we have that for all $x\in \mathbb{R}^n$ and all $y_1,y_2\in \mathbb{R}^m$,
\begin{align*}
\|G_\lambda(x,y_1)&-G_\lambda(x,y_2)\| \notag\\
&=\frac{1}{\lambda}\|\prox_{\lambda f}(x-\lambda \nabla_x H(x,y_2))-\prox_{\lambda f}(x-\lambda \nabla_x H(x,y_1))\| \\
&\leq\|\nabla_x H(x,y_1)-\nabla_x H(x,y_2)\| \\
&\leq L_{12}\|y_1-y_2\|.
\end{align*}
\end{remark}

\begin{lemma}
\label{l:transfer}
Suppose that Assumptions \ref{a:H} and \ref{a:g} hold, $A$ is surjective, and $\mu\in(0,1/\rho)$.
Let $\bar{x} =x^*\in \mathbb{R}^n$ and $\bar{y}=y^*-A^*(AA^*)^{-1}(Ay^*-\prox_{\mu g}(Ay^*))$ with $y^*\in \mathbb{R}^m$. Then
\begin{align*}
&\dist(0,G_\lambda(\bar{x},\bar{y}))+\dist(-\nabla_y H(\bar{x},\bar{y}),A^*\partial g(A\bar{y})) \\
&\leq \dist(0,G_\lambda(x^*,y^*)) +\dist(-\nabla_yH(x^*,y^*),A^*\partial g(\prox_{\mu g}(Ay^*)))\\
&\quad +(L_{12}+L_{22})L_g\sigma_{\min}(A)^{-1}\mu.
\end{align*}
Consequently, if $(x^*,y^*)$ is an $\varepsilon$-approximate stationary point of problem~\eqref{guanghua}, then $(\bar{x},\bar{y})$ is an $\bar{\varepsilon}$-approximate stationary point of problem~\eqref{eq:L1E1} with $\bar{\varepsilon} =\varepsilon +(L_{12}+L_{22})L_g\sigma_{\min}(A)^{-1}\mu$.
\end{lemma}
\begin{proof}
Since $\bar{y}=y^*-A^*(AA^*)^{-1}(Ay^*-\prox_{\mu g}(Ay^*))$, multiplying both sides by $A$ yields $A\bar{y}=\prox_{\mu g}(Ay^*)$. We have that
\begin{align*}
&\dist(-\nabla_yH(\bar{x},\bar{y}),A^*\partial g(A\bar{y})) \\
&\leq \dist(-\nabla_yH(x^*,y^*),A^*\partial g(A\bar{y})) +\|\nabla_yH(x^*,y^*)-\nabla_yH(\bar{x},\bar{y})\| \\
&=\dist(-\nabla_yH(x^*,y^*),A^*\partial g(\prox_{\mu g}(Ay^*))) +\|\nabla_yH(x^*,y^*)-\nabla_yH(x^*,\bar{y})\| \\
&\leq \dist(-\nabla_yH(x^*,y^*),A^*\nabla g_{\mu}(Ay^*)) +L_{22}\|y^* -\bar{y}\|,
\end{align*}
where the last inequality uses $A^*\nabla g_{\mu}(Ay^*)\in A^*\partial g(\prox_{\mu g}(Ay^*))$ (see Lemma~\ref{l:weakcvx}\ref{l:weakcvx_inclu}) and the Lipschitz continuity of $\nabla_y H(x^*,\cdot)$.

Next, by Remark~\ref{r:G(x,.)},
\begin{align*}
\dist(0,G_\lambda(\bar{x},\bar{y})) &= \|G_\lambda(x^*,y^*) -G_\lambda(x^*,y^*) +G_\lambda(\bar{x},\bar{y})\| \\
&\leq \|G_\lambda(x^*,y^*)\| +\|G_\lambda(x^*,y^*) -G_\lambda(\bar{x},\bar{y})\| \\
&= \dist(0,G_\lambda (x^*,y^*)) +\|G_\lambda(x^*,y^*) -G_\lambda(x^*,\bar{y})\| \\
&\leq \dist(0,G_\lambda (x^*,y^*)) +L_{12}\|y^* -\bar{y}\|.
\end{align*}
Note that
\begin{align*}
\|y^* -\bar{y}\| =\|A^*(AA^*)^{-1}(Ay^*-\prox_{\mu g}(Ay^*))\| \leq \sigma_{\min}(A)^{-1}\mu L_g
\end{align*}
due to Lemma~\ref{2.6} and the fact that the operator norm of $A^*(AA^*)^{-1}$ is bounded by the inverse of the smallest singular value $\sigma_{\min}(A)$ of $A$, i.e., $\|A^*(AA^*)^{-1}\|\leq\sigma_{\min}(A)^{-1}$. Altogether, we obtain that
\begin{align*}
&\dist(0,G_\lambda(\bar{x},\bar{y})) +\dist(-\nabla_y H(\bar{x},\bar{y}),A^*\partial g(A\bar{y})) \\
&\leq   \dist(0,G_\lambda(x^*,y^*)) +\dist(-\nabla_yH(x^*,y^*),A^*\partial g(\prox_{\mu g}(Ay^*))) \\
&\quad+(L_{12} +L_{22})L_g\sigma_{\min}(A)^{-1}\mu,
\end{align*}
which completes the proof.
\end{proof}

\section{Proposed algorithm and complexity analysis}
\label{sec:Algorithm}

As introduced earlier, we define an approximate problem of \eqref{eq:L1E1} by the Moreau envelope of $g$ with variable parameter $\mu_k$ as
\begin{align}\label{app}
\min_{(x,y)\in \mathbb{R}^{n}\times\mathbb{R}^{m}}\mathcal{L}_{\mu_k}(x,y)=f(x)+g_{\mu_k}(Ay)+H(x,y).
\end{align}
Denote $\Phi_k(x,y):=g_{\mu_k}(Ay)+H(x,y)$. For the smoothed function $\Phi_k$, we have
\begin{align*}
\nabla_y\Phi_k(x,y) &=\frac{1}{\mu_k}A^{*}(Ay-\prox_{\mu_k g}(Ay))+\nabla_y H(x,y) \text{~~and}\\
\nabla_x\Phi_k(x,y) &=
\nabla_x H(x,y).
\end{align*}
From Assumptions~\ref{a:H} and \ref{a:g}, it can be seen that $\nabla_x\Phi_k(\cdot,y)$ and $\nabla_y\Phi_k(x,\cdot)$ are two Lipschitz continuous functions with Lipschitz constants $L_{11}$ and $L_k:=L_{22}+\|A\|^2\max\{\frac{1}{\mu_k},\frac{\rho}{1-\rho\mu_k}\}$, respectively. Increasing $L_{11}$, $L_{22}$, and/or $\rho$ if necessary, we can assume that $L_{12}^2 <L_{11}(L_{22} +2\rho\|A\|^2)$, which ensures the existence of $\alpha$ in the subsequent algorithm.

Motivated by the works in \cite{BB,BW,LX}, we propose the variable smooth alternating proximal gradient (VsaPG) algorithm to solve \eqref{eq:L1E1}, as detailed in Algorithm~\ref{algo}.

\begin{algorithm}[htb]
\caption{Variable smoothing alternating proximal gradient (VsaPG)}
\label{algo}

\textbf{Step 1.} Let $x^1 =\bar{x}^1\in \mathbb{R}^{n}$, $y^1 =\bar{y}^1\in \mathbb{R}^{m}$, $\mu_1\in (0, 1/\rho)$, and set $k =1$. Let $\gamma\in (0, +\infty)$, $\eta\in (0, +\infty)$, $\theta\in (0, 1)$, $\beta\in(0, 1]$, $\sigma\in (0, 2/L_{11})$, and let $\alpha\in (0, 1)$ be such that
\begin{align*}
1 -\alpha^2 -\frac{L_{12}^2(1 +\alpha)^2}{L_{11}(L_{22} +2\rho\|A\|^2)} >0.
\end{align*}

\textbf{Step 2.} Let $\mu_{k+1}\in [\mu_k/2, \mu_k]$, $\alpha_k\in [-\alpha, \alpha]$, $\beta_k\in [\beta,1]$, $\sigma_k\in [\gamma k^{-\theta}, \sigma]$, and $\tau_k\in [\eta k^{-\theta}, 1/L_k]$. Compute
\begin{align}
\label{s21}&y^{k+1}=\bar{y}^k-\tau_k(A^*\nabla g_{\mu_k}(A\bar{y}^{k})+\nabla_y H(\bar{x}^k,\bar{y}^k)),\\
\label{s41}&\bar{y}^{k+1}=y^{k+1}+\alpha_k(y^{k+1}-\bar{y}^k),\\
\label{s31}&x^{k+1}=\prox_{\sigma_k f}(\bar{x}^k-\sigma_k\nabla_x H(\bar{x}^k,\bar{y}^{k+1})),\\
\label{s42}&\bar{x}^{k+1} =(1 -\beta_k)\bar{x}^k +\beta_kx^{k+1}.
\end{align}

\textbf{Step 3.} If a termination criterion is not met, set $k =k+1$ and go to Step 2.
\end{algorithm}

\begin{remark}[Special cases]
\begin{enumerate}
\item
If $\alpha_k = 0$ and $\beta_k = 1$, then Step~2 of Algorithm~\ref{algo} reduces to
\begin{align*}
&\bar{y}^{k+1}=\bar{y}^k-\tau_k\left(A^*\nabla g_{\mu_k}(A\bar{y}^{k})+\nabla_y H(\bar{x}^k,\bar{y}^k)\right),\\
&\bar{x}^{k+1}=\prox_{\sigma_k f}\left(\bar{x}^k-\sigma_k\nabla_x H(\bar{x}^k,\bar{y}^{k+1})\right).
\end{align*}
\item
If $f(x)\equiv 0$ and $H(x,y)\equiv h(y)$ depends only on $y$, then problem~\eqref{eq:L1E1} reduces to
\[
\min_{y\in\mathbb{R}^m} g(Ay)+h(y).
\]
In this case, when $\alpha_k = 0$ and $\beta_k = 1$, Step~2 of Algorithm~\ref{algo} becomes
\[
\bar{y}^{k+1}=\bar{y}^k-\tau_k\left(A^*\nabla g_{\mu_k}(A\bar{y}^{k})+\nabla h(\bar{y}^k)\right),
\]
which coincides with \cite[Algorithm~1]{BW} if $\mu_k=(2\rho)^{-1}k^{-1/3}$ and $\tau_k=1/L_k$.
\end{enumerate}
\end{remark}

Recalling that $G_\lambda(x,y) =\frac{1}{\lambda}(x -\prox_{\lambda f}(x -\lambda\nabla_x  H(x,y)))$, we have
\begin{align}\label{zhu4.2}
x^{k+1}&=\prox_{\sigma_k f}(\bar{x}^k-\sigma_k\nabla_x H(\bar{x}^k,\bar{y}^{k+1}))\notag \\
&=\bar{x}^k-\sigma_k\left(\frac{1}{\sigma_k}(\bar{x}^k-\prox_{\sigma_k f}(\bar{x}^k-\sigma_k\nabla_x H(\bar{x}^k,\bar{y}^{k+1})))\right)\notag \\
&=\bar{x}^k-\sigma_k G_{\sigma_k}(\bar{x}^k,\bar{y}^{k+1}).
\end{align}
We now arrive at the following lemmas.

\begin{lemma}
\label{l:LL}
Suppose that Assumptions \ref{a:H} and \ref{a:g} hold. Let $(\bar{x}^k,\bar{y}^k)$ be the sequence generated by Algorithm~\ref{algo}. Then
\begin{align*}
\mathcal{L}_{\mu_{k}}(\bar{x}^{k+1},\bar{y}^{k+1}) &\leq \mathcal{L}_{\mu_{k}}(\bar{x}^k,\bar{y}^k)
-\frac{1}{2}\delta\tau_k\|A^*\nabla g_{\mu_k}(A\bar{y}^k)+\nabla_y H(\bar{x}^k,\bar{y}^k)\|^2\notag \\
&\quad -\frac{1}{4}\kappa\sigma_k\|G_{\sigma_k}(\bar{x}^k,\bar{y}^k)\|^2 +(\mu_k-\mu_{k+1})L_g^2,
\end{align*}
where $\delta :=1 -\alpha^2 -\frac{L_{12}^2(1+\alpha)^2}{L_{11}(L_{22} +2\rho\|A\|^2)} >0$ and $\kappa :=\min\{(2 -L_{11}\beta\sigma)\beta, 2 -L_{11}\sigma\} >0$.
\end{lemma}
\begin{proof}
From the Lipschitz continuity of $\nabla_x \Phi_k(\cdot,\bar{y}^{k+1})$ and $\nabla_y\Phi_k(\bar{x}^k, \cdot)$, Lemma \ref{l:descent} implies that
\begin{align*}
\Phi_k(\bar{x}^{k+1},\bar{y}^{k+1}) &\leq\Phi_k(\bar{x}^{k},\bar{y}^{k+1}) +\langle\nabla_x\Phi_k(\bar{x}^k,\bar{y}^{k+1}),\bar{x}^{k+1}-\bar{x}^k\rangle+\frac{L_{11}}{2}\|\bar{x}^{k+1}-\bar{x}^k\|^2, \\
\Phi_k(\bar{x}^k,\bar{y}^{k+1}) &\leq \Phi_k(\bar{x}^k,\bar{y}^k) +\langle\nabla_y\Phi_k(\bar{x}^k,\bar{y}^k),\bar{y}^{k+1}-\bar{y}^k\rangle+\frac{L_k}{2}\|\bar{y}^{k+1}-\bar{y}^k\|^2.
\end{align*}
We note that
\begin{align}\label{eq:grad_y Phi_k}
\nabla_y\Phi_k(\bar{x}^k,\bar{y}^k) =A^*\nabla g_{\mu_k}(A\bar{y}^k)+\nabla_y H(\bar{x}^k,\bar{y}^k)
\end{align}
and from \eqref{s21}, \eqref{s41}, \eqref{s42}, and \eqref{zhu4.2} that
\begin{align*}
\bar{x}^{k+1}-\bar{x}^k &=\beta_k(x^{k+1}-\bar{x}^k)=-\beta_k\sigma_k G_{\sigma_k}(\bar{x}^k,\bar{y}^{k+1}), \\
\bar{y}^{k+1}-\bar{y}^k &=(1+\alpha_k)(y^{k+1}-\bar{y}^k) =-(1+\alpha_k)\tau_k \nabla_y\Phi_k(\bar{x}^k,\bar{y}^k).
\end{align*}
Therefore,
\begin{align}\label{8.7}
&\Phi_k(\bar{x}^{k+1},\bar{y}^{k+1}) \notag \\
&\leq \Phi_k(\bar{x}^k,\bar{y}^k) +\langle\nabla_x\Phi_k(\bar{x}^k,\bar{y}^{k+1}),\bar{x}^{k+1}-\bar{x}^k\rangle+\frac{L_{11}}{2}\|\bar{x}^{k+1}-\bar{x}^k\|^2 \notag \\
&\quad +\langle\nabla_y\Phi_k(\bar{x}^k,\bar{y}^k),\bar{y}^{k+1}-\bar{y}^k\rangle+\frac{L_k}{2}\|\bar{y}^{k+1}-\bar{y}^k\|^2 \notag \\
&= \Phi_k(\bar{x}^k,\bar{y}^k) -\beta_k\sigma_k\langle\nabla_x\Phi_k(\bar{x}^k,\bar{y}^{k+1}),G_{\sigma_k}(\bar{x}^k,\bar{y}^{k+1})\rangle \notag \\
&\quad +\frac{1}{2}L_{11}\beta_k^2\sigma_k^2\|G_{\sigma_k}(\bar{x}^k,\bar{y}^{k+1})\|^2 \notag\\
&\quad -\frac{1}{2}(2 +2\alpha_k -L_k(1+\alpha_k)^2\tau_k)\tau_k\|\nabla_y\Phi_k(\bar{x}^k,\bar{y}^k)\|^2.
\end{align}
By Lemma \ref{l:MoreauEnv} and the fact that $\mu_{k+1}\in [\mu_k/2, \mu_k]$,
\begin{align}\label{8.8}
g_{\mu_{k+1}}(Ay)\leq g_{\mu_k}(Ay) +\frac{\mu_k(\mu_k-\mu_{k+1})}{2\mu_{k+1}}L_g^2\leq g_{\mu_k}(Ay) +(\mu_k-\mu_{k+1})L_g^2.
\end{align}
Adding $H(x,y)$ to both sides of \eqref{8.8} yields
\begin{align}\label{3.4}
\Phi_{k+1}(x,y)\leq \Phi_k(x,y) +(\mu_k-\mu_{k+1})L_g^2.
\end{align}
Letting $x =\bar{x}^{k+1}$, $y =\bar{y}^{k+1}$ in \eqref{3.4} and combining it with \eqref{8.7}, we obtain that
\begin{align}\label{4.4}
&\Phi_{k+1}(\bar{x}^{k+1},\bar{y}^{k+1}) \notag \\
&\leq \Phi_k(\bar{x}^k,\bar{y}^k) -\beta_k\sigma_k\langle\nabla_x\Phi_k(\bar{x}^k,\bar{y}^{k+1}),G_{\sigma_k}(\bar{x}^k,\bar{y}^{k+1})\rangle \notag \\
&\quad +\frac{1}{2}L_{11}\beta_k^2\sigma_k^2\|G_{\sigma_k}(\bar{x}^k,\bar{y}^{k+1})\|^2 \notag\\
&\quad -\frac{1}{2}(2 +2\alpha_k -L_k(1+\alpha_k)^2\tau_k)\tau_k\|\nabla_y\Phi_k(\bar{x}^k,\bar{y}^k)\|^2 +(\mu_k-\mu_{k+1})L_g^2.
\end{align}
By the first-order optimality condition of \eqref{s31},
\begin{align*}
0\in \partial f(x^{k+1})+\frac{1}{\sigma_k}(x^{k+1}-\bar{x}^k+\sigma_k\nabla_x H(\bar{x}^k,\bar{y}^{k+1})),
\end{align*}
or equivalently,
\begin{align*}
\frac{1}{\sigma_k}(\bar{x}^k -x^{k+1}) -\nabla_x H(\bar{x}^k,\bar{y}^{k+1})\in \partial f(x^{k+1}).
\end{align*}
Combining with the convexity of $f$ and \eqref{zhu4.2}, it follows that
\begin{align}\label{8.9}
f(x^{k+1}) &\leq f(\bar{x}^k) -\left\langle \frac{1}{\sigma_k}(\bar{x}^k -x^{k+1}) -\nabla_x H(\bar{x}^k,\bar{y}^{k+1}), \bar{x}^k -x^{k+1}\right\rangle \notag \\
&= f(\bar{x}^k) -\frac{1}{\sigma_k}\|\bar{x}^k -x^{k+1}\|^2 +\langle \nabla_x H(\bar{x}^k,\bar{y}^{k+1}), \bar{x}^k -x^{k+1}\rangle\notag \\
&= f(\bar{x}^k) -\sigma_k\|G_{\sigma_k}(\bar{x}^k,\bar{y}^{k+1})\|^2 +\sigma_k\langle\nabla_x H(\bar{x}^k,\bar{y}^{k+1}),G_{\sigma_k}(\bar{x}^k,\bar{y}^{k+1})\rangle.
\end{align}
Again by the convexity of $f$ and \eqref{s42},
\begin{align*}
f(\bar{x}^{k+1})&=f\big(\beta_k x^{k+1}+(1-\beta_k)\bar{x}^k\big)
\leq \beta_k f(x^{k+1})+(1-\beta_k)f(\bar{x}^k).
\end{align*}
This together with \eqref{8.9} gives
\begin{align}\label{4.5}
f(\bar{x}^{k+1})&\leq f(\bar{x}^k) -\beta_k\sigma_k\|G_{\sigma_k}(\bar{x}^k,\bar{y}^{k+1})\|^2 \notag \\
&\quad +\beta_k\sigma_k\langle\nabla_xH(\bar{x}^k,\bar{y}^{k+1}),G_{\sigma_k}(\bar{x}^k,\bar{y}^{k+1})\rangle.
\end{align}
We now combine \eqref{4.4} and \eqref{4.5} to obtain
\begin{align}\label{4.10}
&\mathcal{L}_{\mu_{k}}(\bar{x}^{k+1},\bar{y}^{k+1}) \notag \\
&\leq \mathcal{L}_{\mu_{k}}(\bar{x}^k,\bar{y}^k) -\frac{1}{2}(2 +2\alpha_k -L_k(1+\alpha_k)^2\tau_k)\tau_k\|\nabla_y\Phi_k(\bar{x}^k,\bar{y}^k)\|^2 \notag \\ &\quad -\frac{1}{2}(2 -L_{11}\beta_k\sigma_k)\beta_k\sigma_k\|G_{\sigma_k}(\bar{x}^k,\bar{y}^{k+1})\|^2 +(\mu_k-\mu_{k+1})L_g^2.
\end{align}
Next, using Remark~\ref{r:G(x,.)},
\begin{align*}
\|G_{\sigma_k}(\bar{x}^k,\bar{y}^k)\|&=\|G_{\sigma_k}(\bar{x}^k,\bar{y}^k)-G_{\sigma_k}(\bar{x}^k,\bar{y}^{k+1})+G_{\sigma_k}(\bar{x}^k,\bar{y}^{k+1})\|\notag \\
&\leq\|G_{\sigma_k}(\bar{x}^k,\bar{y}^k)-G_{\sigma_k}(\bar{x}^k,\bar{y}^{k+1})\|+\|G_{\sigma_k}(\bar{x}^k,\bar{y}^{k+1})\|\notag \\
&\leq L_{12}\|\bar{y}^{k+1}-\bar{y}^k\|+\|G_{\sigma_k}(\bar{x}^k,\bar{y}^{k+1})\|.
\end{align*}
Squaring and applying the inequality $\frac{1}{2}(a+b)^2 \leq a^2 +b^2$, we get
\begin{align*}
\frac{1}{2}\|G_{\sigma_k}(\bar{x}^k,\bar{y}^k)\|^2 &\leq L_{12}^2\|\bar{y}^{k+1}-\bar{y}^k\|^2 +\|G_{\sigma_k}(\bar{x}^k,\bar{y}^{k+1})\|^2 \\
&= L_{12}^2(1+\alpha_k)^2\tau_k^2\|\nabla_y\Phi_k(\bar{x}^k,\bar{y}^k)\|^2 +\|G_{\sigma_k}(\bar{x}^k,\bar{y}^{k+1})\|^2.
\end{align*}
which, combined with \eqref{eq:grad_y Phi_k} and \eqref{4.10}, implies that
\begin{align*}
\mathcal{L}_{\mu_{k}}(\bar{x}^{k+1},\bar{y}^{k+1}) &\leq \mathcal{L}_{\mu_{k}}(\bar{x}^k,\bar{y}^k) -\frac{1}{2}\delta_k\tau_k\|A^*\nabla g_{\mu_k}(A\bar{y}^k)+\nabla_y H(\bar{x}^k,\bar{y}^k)\|^2 \notag \\
&\quad -\frac{1}{4}\kappa_k\sigma_k\|G_{\sigma_k}(\bar{x}^k,\bar{y}^{k+1})\|^2 +(\mu_k-\mu_{k+1})L_g^2,
\end{align*}
where
\begin{align*}
\delta_k &:=2 +2\alpha_k -L_k(1+\alpha_k)^2\tau_k -(2 -L_{11}\beta_k\sigma_k)\sigma_k\beta_kL_{12}^2(1+\alpha_k)^2\tau_k, \\
\kappa_k &:=(2 -L_{11}\beta_k\sigma_k)\beta_k.
\end{align*}
Noting that $\tau_k\leq 1/L_k$ and that
\begin{align*}
(2 -L_{11}\beta_k\sigma_k)\sigma_k\beta_k\leq \frac{1}{L_{11}}\left(\frac{2 -L_{11}\beta_k\sigma_k +L_{11}\beta_k\sigma_k}{2}\right)^2 =\frac{1}{L_{11}},
\end{align*}
we have
\begin{align*}
\delta_k &\geq 2 +2\alpha_k -(1+\alpha_k)^2 -\frac{L_{12}^2}{L_{11}L_k}(1+\alpha_k)^2 \\
&= 1 -\alpha_k^2 -\frac{L_{12}^2}{L_{11}L_k}(1+\alpha_k)^2 \\
&\geq \delta =1 -\alpha^2 -\frac{L_{12}^2(1+\alpha)^2}{L_{11}(L_{22} +2\rho\|A\|^2)},
\end{align*}
where the last inequality follows from from the fact that $\alpha_k\in [-\alpha, \alpha]\subseteq [-1, \alpha]$ and that $L_k =L_{22} +\|A\|^2\max\{\frac{1}{\mu_k}, \frac{\rho}{1 -\rho\mu_k}\}\geq L_{22} +2\rho\|A\|^2$ since
\begin{align*}
\max\left\{\frac{1}{\mu_k}, \frac{\rho}{1 -\rho\mu_k}\right\} =\begin{cases}
\frac{1}{\mu_k} &\text{if~} 0 < \mu_k \leq \frac{1}{2\rho}, \\
\frac{\rho}{1 -\rho\mu_k} &\text{if~} \frac{1}{2\rho} < \mu_k < \frac{1}{\rho}.
\end{cases}
\end{align*}
Finally, we derive from $\sigma_k\in \sigma$ and $\beta_k\in [\beta, 1]$ that
\begin{align*}
\kappa_k =(2 -L_{11}\beta_k\sigma_k)\beta_k \geq (2 -L_{11}\beta_k\sigma)\beta_k \geq \kappa =\min\{(2 -L_{11}\beta\sigma)\beta, 2 -L_{11}\sigma\}
\end{align*}
as $(2 -L_{11}\beta_k\sigma)\beta_k$ is a quadratic function of $\beta_k$ with leading coefficient $-L_{11}\sigma <0$. This completes the proof.
\end{proof}

Obviously, Assumption \ref{a:g} implies that $\liminf_{k\rightarrow \infty} \mathcal{L}_{\mu_k}(x^k,y^k)>-\infty$. From now on, we denote
\begin{align*}
\mathcal{L}^* :=\liminf_{k\rightarrow \infty} \mathcal{L}_{\mu_k}(x^k,y^k).
\end{align*}

\begin{lemma}
\label{l:e-stationary}
Suppose that Assumptions~\ref{a:H} and \ref{a:g} hold. Let $(\bar{x}^k,\bar{y}^k)$ be the sequence generated by Algorithm~\ref{algo}. Then
\begin{align*}
\min_{1\leq j\leq k}\left(\|A^*\nabla g_{\mu_j}(A\bar{y}^j)+\nabla_y H(\bar{x}^j,\bar{y}^j)\|+\|G_{\sigma_j}(\bar{x}^j,\bar{y}^j)\|\right)
\leq \Theta k^{\frac{\theta-1}{2}},
\end{align*}
where $\Theta :=\sqrt{2(\ln 2)^{-1}M^{-1}(\mathcal{L}_{\mu_1}(\bar{x}^1,\bar{y}^1)-\mathcal{L}^*+\mu_1L_g^2)}$ and
\begin{align*}
M :=\min\left\{\frac{(1 -\alpha^2)\eta}{2} -\frac{L_{12}^2(1+\alpha)^2\eta}{2L_{11}(L_{22} +2\rho\|A\|^2)}, \frac{(2 -L_{11}\beta\sigma)\beta\gamma}{4}, \frac{(2 -L_{11}\sigma)\gamma}{4}\right\}.
\end{align*}
\end{lemma}
\begin{proof}
According to Lemma~\ref{l:LL},
\begin{align}\label{4.10'}
\mathcal{L}_{\mu_{k}}(\bar{x}^{k+1},\bar{y}^{k+1}) &\leq \mathcal{L}_{\mu_{k}}(\bar{x}^k,\bar{y}^k) -\frac{1}{2}\delta\tau_k\|A^*\nabla g_{\mu_k}(A\bar{y}^k)+\nabla_y H(\bar{x}^k,\bar{y}^k)\|^2\notag \\
&\quad -\frac{1}{4}\kappa\sigma_k\|G_{\sigma_k}(\bar{x}^k,\bar{y}^k)\|^2 +(\mu_k-\mu_{k+1})L_g^2.
\end{align}
Summing \eqref{4.10'} from $k=1$ to $k=K$, we have
\begin{align}\label{4.11}
&\sum_{k=1}^K \left(\frac{1}{2}\delta\tau_k\|A^*\nabla g_{\mu_k}(A\bar{y}^k)+\nabla_y H(\bar{x}^k,\bar{y}^k)\|^2 +\frac{1}{4}\kappa\sigma_k\|G_{\sigma_k}(\bar{x}^k,\bar{y}^k)\|^2\right) \notag \\
&\leq \mathcal{L}_{\mu_1}(\bar{x}^1,\bar{y}^1)-\mathcal{L}_{\mu_{K+1}}(\bar{x}^{K+1},\bar{y}^{K+1})+(\mu_1-\mu_{K+1})L_g^2\notag \\
&\leq \mathcal{L}_{\mu_1}(\bar{x}^1,\bar{y}^1)-\mathcal{L}^*+\mu_1L_g^2.
\end{align}
Since $\sigma_k \geq\gamma k^{-\theta}$, $\tau_k\geq\eta k^{-\theta}$, and $M =\min\{\frac{1}{2}\delta\eta, \frac{1}{4}\kappa\gamma\}$, it follows from \eqref{4.11} that
\begin{multline*}
M\left(\sum_{k=1}^{K}k^{-\theta}\right) \min_{1\leq j\leq K}(\|A^*\nabla g_{\mu_j}(A\bar{y}^j)+\nabla_yH(\bar{x}^j,\bar{y}^j)\|^2+
\|G_{\sigma_j}(\bar{x}^j,\bar{y}^j)\|^2) \\ \leq\mathcal{L}_{\mu_1}(\bar{x}^1,\bar{y}^1)
-\mathcal{L}^*+\mu_1L_g^2,
\end{multline*}
Note that
\begin{align*}
\sum_{k=1}^Kk^{-\theta}&\geq \sum_{k=1}^K\int_k^{k+1}x^{-\theta}dx=\int_1^{K+1}x^{-\theta}dx=\frac{1}{1-\theta}((1+K)^{1-\theta}-1)\\
&\geq(\ln 2)K^{1-\theta},
\end{align*}
where the last inequality is obtained by applying Lemma \ref{l:estimate} with $x =K$ and $\alpha =1 -\theta$. We deduce that
\begin{multline*}
\min_{1\leq j\leq K}(\|A^*\nabla g_{\mu_j}(A\bar{y}^j)+\nabla_yH(\bar{x}^j,\bar{y}^j)\|^2+\|G_{\sigma_j}(\bar{x}^j,\bar{y}^j)\|^2) \\ \leq (\ln 2)^{-1}M^{-1}(\mathcal{L}_{\mu_1}(\bar{x}^1,\bar{y}^1)-\mathcal{L}^*+\mu_1L_g^2)k^{\theta-1}.
\end{multline*}
Using the inequality $(\|a\|+\|b\|)^2\leq2(\|a\|^2+\|b\|^2)$, we have
\begin{multline*}
\min_{1\leq j\leq k}(\|A^*\nabla g_{\mu_j}(A\bar{y}^j)+\nabla_yH(\bar{x}^j,\bar{y}^j)\|+\|G_{\sigma_j}(\bar{x}^j,\bar{y}^j)\|) \\ \leq \sqrt{2(\ln 2)^{-1}M^{-1}(\mathcal{L}_{\mu_1}(\bar{x}^1,\bar{y}^1)-\mathcal{L}^*+\mu_1L_g^2)} k^{\frac{\theta-1}{2}}.
\end{multline*}
The proof is complete.
\end{proof}

From Lemma~\ref{l:e-stationary}, we know that Algorithm~\ref{algo} achieves an $\varepsilon$-approximate solution to problem \eqref{guanghua} with a complexity of $\mathcal{O}(\varepsilon^{-\frac{2}{1-\theta}})$. We now use this result to analyze the complexity of Algorithm~\ref{algo} for solving problem \eqref{eq:L1E1}.

\begin{theorem}
Suppose that Assumptions~\ref{a:H} and \ref{a:g} hold. Let $(\bar{x}^k,\bar{y}^k)$ be the sequence generated by Algorithm~\ref{algo} with $\mu_k =\mu_1 k^{-\min\{\theta, \frac{1-\theta}{2}\}}$. Then, for $\hat{x}^k=\bar{x}^k$ and $\hat{y}^k=\bar{y}^k-A^*(AA^*)^{-1}(A\bar{y}^k-\prox_{\mu g}(A\bar{y}^k))$, it holds that
\begin{multline*}
\min_{1\leq j\leq k}(\dist(-\nabla_yH(\hat{x}^j,\hat{y}^j),A^*\partial g(A\hat{y}^j))+\dist(0,G_{\sigma_j}(\hat{x}^j,\hat{y}^j))) \\ \leq (\Theta +(L_{12}+L_{12})L_g\sigma_{\min}(A)^{-1}\mu_1)k^{-\min\{\theta, \frac{1-\theta}{2}\}},
\end{multline*}
where $\Theta :=\sqrt{2(\ln 2)^{-1}M^{-1}(\mathcal{L}_{\mu_1}(\bar{x}^1,\bar{y}^1)-\mathcal{L}^*+\mu_1L_g^2)}$ and
\begin{align*}
M :=\min\left\{\frac{(1 -\alpha^2)\eta}{2} -\frac{L_{12}^2(1+\alpha)^2\eta}{2L_{11}(L_{22} +2\rho\|A\|^2)}, \frac{(2 -L_{11}\beta\sigma)\beta\gamma}{4}, \frac{(2 -L_{11}\sigma)\gamma}{4}\right\}.
\end{align*}
\end{theorem}
\begin{proof}
Using Lemmas~\ref{l:e-stationary} and \ref{l:weakcvx}, we obtain
\begin{align*}
&\min_{1\leq j\leq k}(\dist(-\nabla_yH(\bar{x}^j,\bar{y}^j),A^*\partial g(\prox_{\mu_j g}(A\bar{y}^j))+\dist(0,G_{\sigma_j}(\bar{x}^j,\bar{y}^j))) \\
&\leq\min_{1\leq j\leq k}(\|A^*\nabla g_{\mu_j}(A\bar{y}^j)+\nabla_yH(\bar{x}^j,\bar{y}^j)\|+\|G_{\sigma_j}(\bar{x}^j,\bar{y}^j)\|)\leq \Theta k^{\frac{\theta-1}{2}},
\end{align*}
which together with Lemma~\ref{l:transfer} implies that
\begin{align*}
&\min_{1\leq j\leq k}(\dist(-\nabla_yH(\hat{x}^j,\hat{y}^j),A^*\partial g(A\hat{y}^j))+\dist(0,G_{\sigma_j}(\hat{x}^j,\hat{y}^j))) \\
&\leq \Theta k^{\frac{\theta-1}{2}} +(L_{12}+L_{12})L_g\sigma_{\min}(A)^{-1}\mu_k \\
&= \Theta k^{-\frac{1-\theta}{2}} +(L_{12}+L_{12})L_g\sigma_{\min}(A)^{-1}\mu_1 k^{-\min\{\theta, \frac{1-\theta}{2}\}} \\
&\leq (\Theta +(L_{12}+L_{12})L_g\sigma_{\min}(A)^{-1}\mu_1)k^{-\min\{\theta, \frac{1-\theta}{2}\}}.
\end{align*}
The proof is complete.
\end{proof}

\begin{remark}
When $\theta \in (0,\frac{1}{3})$, Algorithm~\ref{algo} achieves an $\varepsilon$-approximate solution to problem~\eqref{eq:L1E1} with complexity $\mathcal{O}(\varepsilon^{-\frac{1}{\theta}})$. For $\theta \in (\frac{1}{3},1)$, the complexity becomes $\mathcal{O}(\varepsilon^{-\frac{2}{1-\theta}})$. In the critical case $\theta = \frac{1}{3}$, the algorithm requires $\mathcal{O}(\varepsilon^{-3})$ iterations, which is optimal. The same complexity has been established in \cite{BW,LX} exclusively for $\theta = \frac{1}{3}$.
\end{remark}

\section{Numerical experiments}
\label{sec:Numerical}

In this section, we provide two numerical examples to compare the performance of our VsaPG algorithm (Algorithm~\ref{algo}) with several existing methods, including PALM presented in \cite{BST}, iPALM presented in \cite{PS}, GiPALM presented in \cite{GCH}, and NiPALM presented in \cite{WH}. All codes are run under MATLAB R2018a and Windows 10 system, Intel(R) Core(TM) i5-8250U CPU @ 1.60GHz. ``Iter'' represents the number of iterations, ``Time'' represents the running time, ``err'' is the error, and ``res'' represents the residual.

\begin{example}
\label{ex:signal recovery}
Consider the following sparse signal recovery problem \cite{DD}:
\begin{align}\label{6.18}
\min_{x\in \mathbb{R}^n} \|x\|_0 \text{~~subject to~~} Cx=b.
\end{align}
Here, $C\in \mathbb{R}^{m\times n}$ is a sampling matrix, $b\in \mathbb{R}^{m}$ is an observation and $x$ is the signal we would like to recover.
\end{example}

It is well known that problem \eqref{6.18} can be solved by $L_1$ regularization. However, this method may lead to bias due to the proximal operator of the 1-norm does not approach the identity even for large arguments. For this reason, nonconvex alternatives to $\|\cdot\|_1$ are often used to reduce bias including several weakly convex regularizers. In this example, we choose the minimax concave penalty (MCP), introduced in \cite{Z} and used in \cite{SG,LP}, which is a family of functions $r_{\lambda,\xi}:\mathbb{R}\rightarrow\mathbb{R}_+$ with $\xi>0$ and $\lambda>0$, and defined by
\begin{align}\label{mcp}
r_{\lambda,\xi}(z)=
\begin{cases}
\lambda|z|-\frac{z^2}{2\xi},&|z|\leq\xi\lambda,\\
\frac{\xi{\lambda}^2}{2},&\text{otherwise}.
\end{cases}
\end{align}
It is easy to see that this function is $\rho$-weak convexity with $\rho=\xi^{-1}$. The proximal operator of this function can be written as follows (see \cite{BI}):
\begin{align}\label{proxmcp}
\prox_{\gamma,r_{\lambda,\xi}}(z)=
\begin{cases}
0,&|z|\leq\gamma\lambda,\\
\frac{z-\lambda\gamma \mbox{sgn}(z)}{1-(\gamma/\xi)},&\gamma\leq|z|\leq\xi\lambda,\\
z,&|z|\geq\xi\lambda.
\end{cases}
\end{align}
Therefore, problem \eqref{6.18} can be solved by the following transformed form:
\begin{align}\label{6.181}
\min_{x\in \mathbb{R}^n}\frac{1}{2}\|Cx-b\|_2^2+\sum_{i=1}^m r_{\lambda,\xi}(x_i).
\end{align}

To solve problem \eqref{6.181}, we introduce a new variable $y\in \mathbb{R}^n$, then model \eqref{6.181} is transformed into:
\begin{align}\label{4.2}
\min_{(x,y)\in\mathbb{R}^n\times\mathbb{R}^n}\frac{1}{2}\|Cx-b\|_2^2+\sum_{i=1}^nr_{\lambda,\xi}(y_i)+\frac{\mu}{2}\|x-y\|_2^2,
\end{align}
where $\mu>0$ is a penalty parameter. Model \eqref{4.2} satisfies the form of problem \eqref{eq:L1E1} when we set $f(x)=\frac{1}{2}\|Cx-b\|_2^2$, $g(Ay)=\sum_{i=1}^nr_{\lambda,\xi}(y_i)$ (where $A$ is equal to the identity matrix $I$) and $H(x,y)=\frac{\mu}{2}\|x-y\|_2^2$.

For model \eqref{4.2}, each element of $C$ is taken from a standard normal distribution, and then all columns of $C$ are normalized. We generate a random sparse vector $x$ in $\mathbb{R}^n$ with a sparsity of 0.03, where the non-zero entries are drawn from $N(0,1)$.   The noise vector $\omega \thicksim N(0,10^{-3}I)$, $b=Cx+\omega$, and the regularization parameter $\lambda=0.01\|C^Tb\|_\infty$. The residual at iteration $k$ is defined as $r^k=x^k-y^k$, and the stopping criterion for all algorithms in the experiment is
\begin{align*}
\frac{\|r^k\|}{\mbox{max}\{\|x^k\|,\|y^k\|\}}<err, \;\text{or} \;maxiter=5000.
\end{align*}

The parameters are set as follows:

VsaPG: $\tau_k=\frac{1}{L_k}=\frac{1}{L_{22}+\mu_k^{-1}}$, $\sigma_k=\frac{1}{L_{11}}$, $L_{22}=5$, $L_{11}=5$ , $\mu=5$, $\alpha_k=0.2$ and $\beta_k=0.99$.

PALM: $c_k=18$, $d_k=18$ and $\mu=5.$

iPALM: $c_k=18$, $d_k=18$, $\mu=5$, $\alpha_k=0.2$, $\beta_k=0.2$, $\tilde{\alpha}_k=0.2$ and $\tilde{\beta}_k=0.2$.

GiPALM: $c_k=18$, $d_k=18$, $\mu=5$, $\alpha=0.2$ and $\beta=0.2.$

NiPALM: $c_k=18$, $d_k=18$, $\mu=5$, $\alpha_k=0.2$, $\beta_k=0.2$, $\tilde{\alpha}_k=0.2$ and $\tilde{\beta}_k=0.2$, $\alpha=0.2$ and $\beta=0.2.$

Table \ref{tab1} reports the number of iterations and CPU time of VsaPG, PALM, iPALM, GiPALM, and NiPALM under identical dimensionality settings across varying error levels. Tables \ref{tab2} and \ref{tab3} compare the iteration counts and CPU time of VsaPG, PALM, iPALM, GiPALM, and NiPALM under fixed error levels across varying dimensions. Figures \ref{fig1} compare the objective values and residuals of VsaPG,  PALM, iPALM, GiPALM, and NiPALM under the configuration $err=10^{-3}$, $m=128$, and $n=512$. The results demonstrate that VsaPG outperform PALM, iPALM, GiPALM, and NiPALM in both iteration count and CPU time. Specifically, VsaPG   achieve a runtime reduction of at least 20\% compared to PALM, iPALM, GiPALM, and NiPALM.

\begin{table}[htbp]
  \centering
  \caption{Comparison of different algorithms regarding different errors when $m = 128$ and $n = 512$}
  \renewcommand{\arraystretch}{1.5}
  \resizebox{\textwidth}{!}{
  \begin{tabular}{ccccccccccc}
    \toprule
    \multirow{2}*{Algorithm} & \multicolumn{2}{c}{\(err=10^{-2}\)} &\multicolumn{2}{c}{\(err=10^{-3}\)} & \multicolumn{2}{c}{\(err=10^{-4}\)}& \multicolumn{2}{c}{\(err=10^{-5}\)}& \multicolumn{2}{c}{\(err=10^{-6}\)} \\\cline{2-11}
                 & Iter & Time    & Iter & Time   & Iter & Time   &Iter &Time   &Iter &Time\\
    \midrule
    VsaPG        & \textbf{11}   & \textbf{0.0012}  & \textbf{23}   & \textbf{0.0165} & \textbf{40}   & \textbf{0.0173} &\textbf{103}   &\textbf{0.0249} &\bf{124}  &\bf{0.0340}\\
    PALM         & 32        & 0.0384       & 76        & 0.0318      & 179       & 0.0464      &309        &0.1974      &683       &0.2409\\
    iPALM        & 26        & 0.0174       & 62        & 0.0308      & 149       & 0.0380      &551        &0.1308      &596       &0.1853\\
    GiPALM       & 26        & 0.0163       & 61        & 0.0258      & 144       & 0.0362      &248        &0.0828      &295       &0.0480 \\
    NiPALM       & 25        & 0.0081       & 60        & 0.0298      & 144       & 0.0255      &422        &0.1194      &465       &0.0704\\
    \bottomrule
  \end{tabular}
  \label{tab1}
  }
\end{table}

\begin{table}[htbp]
  \centering
  \caption{Comparison of different algorithms regarding different dimensions when $err = 10^{-6}$}
  \renewcommand{\arraystretch}{1.5}
  \resizebox{\textwidth}{!}{
  \begin{tabular}{ccccccccc}
    \toprule
    \multirow{2}*{Algorithm} & \multicolumn{2}{c}{\(m=128,n=512\)} & \multicolumn{2}{c}{\(m=256,n=1024\)} & \multicolumn{2}{c}{\(m=512,n=2048\)} & \multicolumn{2}{c}{\(m=1024,n=4096\)}\\ \cline{2-9}
                 & Iter & Time    & Iter & Time    & Iter  & Time     & Iter & Time      \\
    \midrule
    VsaPG        & \textbf{124}  &\textbf{0.0340}   & \textbf{124}  & \textbf{0.1376}  & \textbf{126}   & \textbf{0.3954}   & \textbf{256}  & \textbf{2.6187}  \\
    PLAM         & 683       &0.2409        & 368       & 0.2826       & 397        & 1.3873        & 450       & 6.0785  \\
    iPALM        & 596       &0.1853        & 306       & 0.2363       & 335        & 1.2484        & 394       & 5.2087  \\
    GiPALM       & 295       &0.0480        & 297       & 0.2036       & 523        & 1.9426        & 361       & 4.7468  \\
    NiPALM       & 465       &0.0704        & 296       & 0.1791       & 323        & 0.8845        & 369       & 4.7908 \\
    \bottomrule
  \end{tabular}
  \label{tab2}
  }
\end{table}

\begin{table}[htbp]
    \centering
    \caption{Comparison of different algorithms regarding different dimensions when $err=10^{-6}$}
    \renewcommand{\arraystretch}{2.0}
  \resizebox{\textwidth}{!}{

            \begin{tabular}{ccccccccccccc}
                \hline
                \multirow{2}*{Algorithm} & \multicolumn{3}{c}{\(m=500,n=1000\)} & \multicolumn{3}{c}{\(m=1000,n=2000\)} & \multicolumn{3}{c}{\(m=3000,n=6000\)} & \multicolumn{3}{c}{\(m=4000,n=8000\)}\\ \cline{2-12}
                             & Iter & Time   &res        & Iter & Time    &res        & Iter  & Time     &res        & Iter & Time     &res     \\ \hline
                VsaPG   & \textbf{539}  & \textbf{0.9638} &\textbf{2.2244e-05} & \textbf{632}  & \textbf{3.5121}  &\textbf{3.1733e-05} & \textbf{690}  & \textbf{32.1651}  &5.5627e-05 & \bf{1011}  & \bf{84.1016}  &6.3528e-05  \\
                PLAM         & 1554 & 2.3190 &2.2287e-05 & 1912 & 10.8349 &3.1946e-05 & 2569  & 122.4736 &5.5597e-05 & 2755 & 231.7659 &6.3511e-05  \\
                iPALM        & 1291 & 2.0363 &2.2261e-05 & 1578 & 8.9768  &3.1796e-05 & 2144  & 101.9632 &5.5480e-05 & 2227 & 188.2360 &\textbf{6.3250e-05}  \\
                GiPALM       & 1251 & 1.9741 &2.2283e-05 & 1537 & 8.8203  &3.1834e-05 & 1962  & 94.9705  &\textbf{5.5446e-05} & 2523 & 213.9726 &6.3355e-05  \\
                NiPALM       & 1247 & 1.8997 &2.2345e-05 & 1542 & 8.8454  &3.1955e-05 & 1922  & 90.0365  &5.5467e-05 & 2038 & 167.5054 &6.3506e-05  \\ \hline
            \end{tabular}
            \label{tab3}

    }
\end{table}

\begin{figure}[htbp!]
    \centering
    \begin{minipage}[t]{0.48\textwidth}
        \centering
        \includegraphics[width=\textwidth]{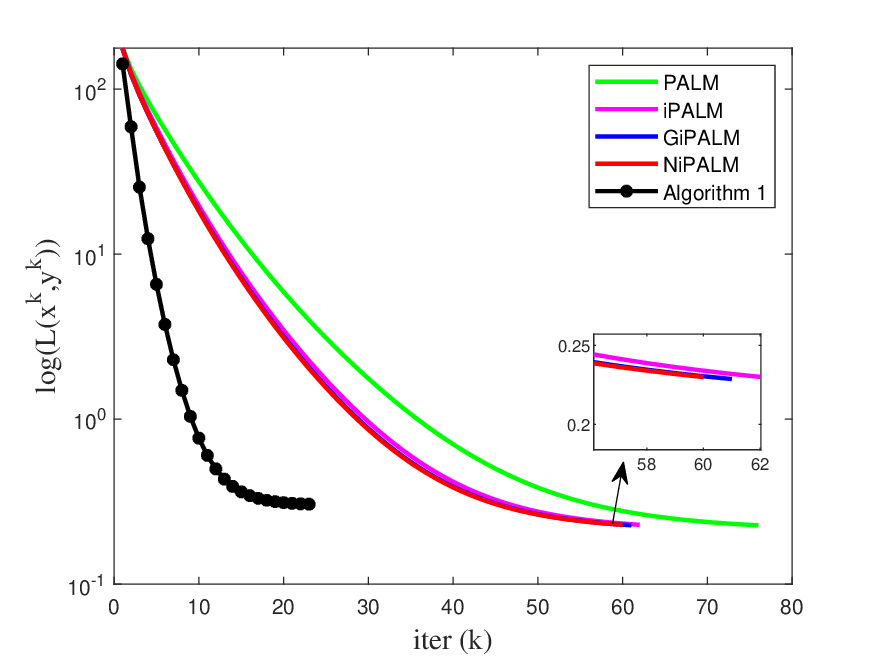}
        \caption*{ (a) Objective value}
    \end{minipage}
    \hfill
    \begin{minipage}[t]{0.48\textwidth}
        \centering
        \includegraphics[width=\textwidth]{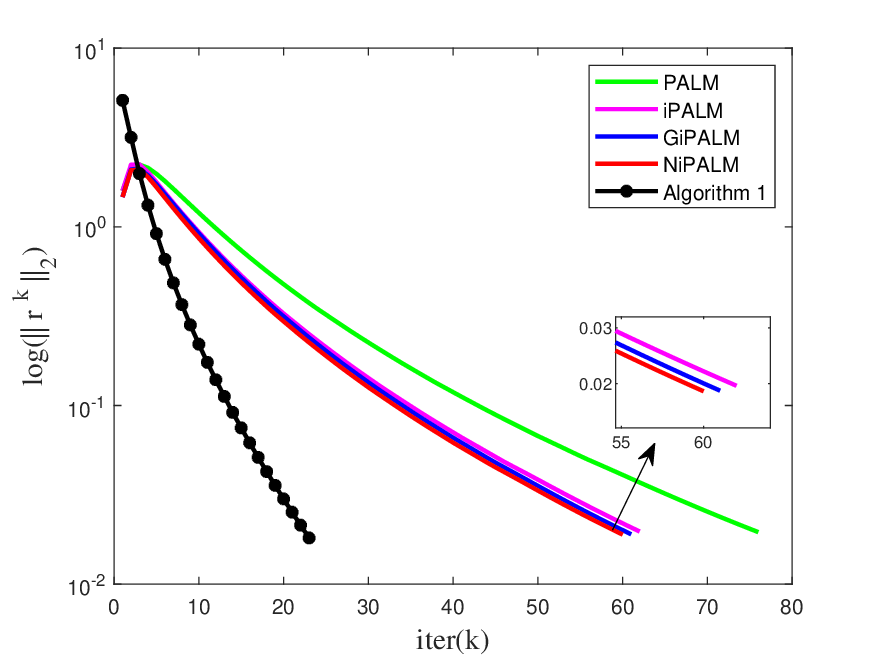}
        \caption*{ (b) residual}
    \end{minipage}

    \caption{Objective value and residual when $err=10^{-3}$, $m=128$ and $n=512$}
    \label{fig1}
\end{figure}

\begin{remark}
As discussed above, we solve problem \eqref{6.181} by transforming it into problem \eqref{4.2}. In fact, the proximal gradient method (abbreviated as PG) can also directly solve problem \eqref{6.181}, with the iterative formula being:
\begin{align*}
x^{k+1}\in \argmin_{x\in \mathbb{R}^n}\{g(x)+\frac{c_k}{2}\|x-x^k\|^2+\langle x-x^k,\nabla f(x^k)\rangle \},
\end{align*}
where $f(x)=\frac{1}{2}\|Ax-b\|_2^2$ and $g(x)=\sum_{i=1}^nr_{\lambda,\xi}(x_i)$.
\end{remark}

Below we will compare VsaPG with PG, the residual at the $k$th iteration is denoted by $r^k=x^k-x^{k-1}$, and the stopping criterion is $\frac{\|r^k\|_2}{\max\{\|x^{k-1}\|,\|x^k\|\}}\leq err$.
The parameter configuration is consistent with the previous settings. The objective valued and residual are shown in Figure \ref{fig2}. The iteration counts and computational time of VsaPG, and PG are summarized in Tables \ref{tab4} and \ref{tab5}. The results demonstrate that VsaPG exhibit superior performance to PG in both iteration efficiency and computational time.
\begin{figure}[htbp!]
    \centering
    \begin{minipage}[t]{0.48\textwidth}
        \centering
        \includegraphics[width=\textwidth]{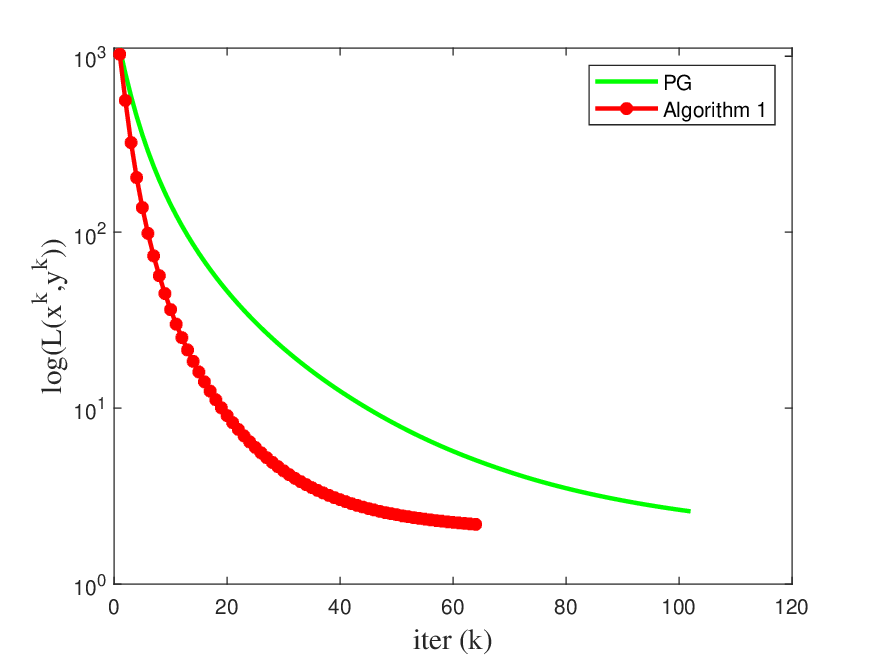}

        \caption*{(a) Objective value}
    \end{minipage}
    \hfill
    \begin{minipage}[t]{0.48\textwidth}
        \centering
        \includegraphics[width=\textwidth]{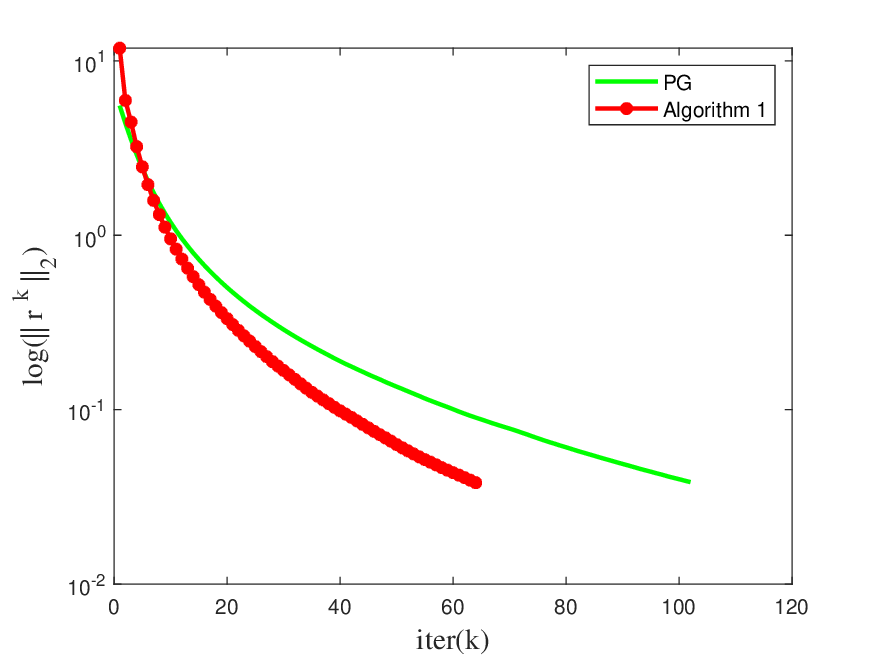}
        \caption*{(b) residual}
    \end{minipage}
    \caption{Objective value and residual when $err=10^{-3}$, $m=1500$ and $n=3000$}
    \label{fig2}
\end{figure}

\begin{table}[htbp]
  \centering
  \caption{Comparison of different algorithms regarding different errors when $m = 1500$ and $n = 3000$}
  \renewcommand{\arraystretch}{1.5}
  \resizebox{\textwidth}{!}{
  \begin{tabular}{ccccccccccc}

    \toprule
    \multirow{2}*{Algorithm} & \multicolumn{2}{c}{\(err=10^{-2}\)} &\multicolumn{2}{c}{\(err=10^{-3}\)} & \multicolumn{2}{c}{\(err=10^{-4}\)}& \multicolumn{2}{c}{\(err=10^{-5}\)}& \multicolumn{2}{c}{\(err=10^{-6}\)} \\\cline{2-11}
                 & Iter & Time    & Iter & Time   & Iter  & Time    &Iter  &Time    &Iter &Time\\
    \midrule
    VsaPG        & \textbf{18}   & \textbf{0.2270}  & \textbf{64}   & \textbf{0.7987} & \textbf{171}   & \textbf{2.1622}  &\textbf{390}   &\textbf{4.8333}  &\textbf{531}  &\textbf{6.6481}\\
    PG           & 24   & 1.0166  & 102  & 4.2090 & 297   & 12.2678 &807   &33.7639 &1562  &64.9002\\
    \bottomrule
  \end{tabular}
  \label{tab4}
  }
\end{table}

\begin{table}[htbp]
  \centering
  \caption{Comparison of different algorithms regarding different dimensions when $err = 10^{-6}$}
  \renewcommand{\arraystretch}{1.5}
  \resizebox{\textwidth}{!}{
  \begin{tabular}{ccccccccc}
    \toprule
    \multirow{2}*{Algorithm} & \multicolumn{2}{c}{\(m=500,n=1000\)} &\multicolumn{2}{c}{\(m=1000,2000\)} & \multicolumn{2}{c}{\(m=1500,n=3000\)}& \multicolumn{2}{c}{\(m=3000,6000\)} \\\cline{2-9}
                 & Iter  & Time    & Iter  & Time    &Iter   &Time       & Iter   & Time     \\
    \midrule
    VsaPG       & \textbf{452}   & \textbf{0.6723}  & \textbf{538}   &\textbf {3.3049}  &\textbf{531}    &\textbf{6.5897}     & \textbf{591}    & \textbf{28.3026}  \\
    PG           & 756   & 3.5344  & 750   & 13.6172 &1562   &64.9002    & 1038   & 164.1311 \\
    \bottomrule
  \end{tabular}
  \label{tab5}
  }
\end{table}

\begin{example}
We consider the following image denoising problem:
\begin{align}\label{5.1}
\min_{x\in \mathbb{R}^{n}} \frac{1}{2}\|x-\varepsilon\|_2^2+\lambda\|\nabla x\|_1,
\end{align}
where $\nabla x\in \mathbb{R}^{n}$ represents the discrete gradient of image $x\in \mathbb{R}^{n}$, $\varepsilon \in \mathbb{R}^{n}$ represents the input noisy image and $\lambda> 0$ is a regularization parameter.
\end{example}

Similar to Example~\ref{ex:signal recovery}, we adopt weakly convex regularizations $\sum_{i=1}^nr_{\lambda,\xi}(\nabla x_i)$ which serves as a nonconvex alternative to $\lambda\|\nabla x\|_1$. The explicit formulation of $r_{\lambda,\xi}$ and its corresponding proximal operator $\prox_{r_{\lambda,\xi}}$ are defined in \eqref{mcp} and \eqref{proxmcp}, respectively.

We consider a new variable $y\in \mathbb{R}^{n^2}$ and then transform problem \eqref{5.1} into the following problem:
\begin{align}\label{5.2}
\min_{x\in \mathbb{R}^{n^2},y\in \mathbb{R}^{n^2}} \frac{1}{2}\|x-\varepsilon\|_2^2+\sum_{i=1}^nr_{\lambda,\xi}(y_i)+\frac{\mu}{2}\|y-\nabla x\|_2^2,
\end{align}
where $\mu\geq0$ is a penalty parameter. Let $f(x)=\frac{1}{2}\|x-\varepsilon\|_2^2$, $g(Ay)=\sum_{i=1}^nr_{\lambda,\xi}(y_i)$ (where $A$ is equal to the identity matrix $I$), and $H(x,y)=\frac{\mu}{2}\|y-\nabla x\|_2^2$, then \eqref{5.2} satisfies the form of problem \eqref{eq:L1E1}.

Next, we test three images named $boy$, $Cameraman$ and $peppers$, respectively. These images are added with Gaussian white noise with zero mean and a standard deviation of 0.01. The stopping criterion for all algorithms is defined as
\begin{align*}
\|(y^{k+1},\nabla x^{k+1})-(y^k,\nabla x^k)\|<err, \;\text{or}\;maxiter=500, \text{where} \; err=10^{-2}.
\end{align*}

The parameters  are set as follows:

VsaPG: $\tau_k=\frac{1}{L_k}=\frac{1}{L_{22}+\mu_k^{-1}}$, $\sigma_k=\frac{1}{L_{11}}$, $L_{22}=5$ ,$L_{11}=5$, $\mu=1$, $\alpha_k=0.2$ and $\beta_k=0.99.$

PALM: $c_k=30$, $d_k=30$ and $\mu=1.$

GiPALM: $c_k=30$, $d_k=30$, $\mu=1$, $\alpha=0.2$ and $\beta=0.2.$

NiPALM: $c_k=30$, $d_k=30$, $\mu=1$, $\alpha_k=0.2$, $\beta_k=0.2$, $\tilde{\alpha}_k=0.2$ , $\tilde{\beta}_k=0.2$, $\alpha=0.2$ and $\beta=0.2.$

Typically, we use signal-to-noise ratio (SNR) as a measurement of denoising quality. SNR is defined by
\begin{align*}
SNR=20\log_{10}\frac{\|x^*\|_2}{\|x-x^*\|_2}.
\end{align*}
where $x^*$ and $x$ represent the original image and the restored image, respectively.

The original clean and the noisy images are shown in Figure \ref{fig3}. Table \ref{tab6} records the number of iterations, CPU time and SNR values for VsaPG, PALM, GiPALM and NiPALM for different image restoration tasks. The three denoising images recovered by VsaPG, PALM, GiPALM and NiPALM are shown in Figure \ref{fig4}. The evolution of SNR are shown in Figure \ref{fig5}.
The results demonstrate that VsaPG achieve higher efficiency than PALM, iPALM, GiPALM, and NiPALM in both CPU time and SNR.

\begin{figure}[htbp!]
	\centering
	\subfigure[Boy original ]{
		\includegraphics[width=4.3cm]{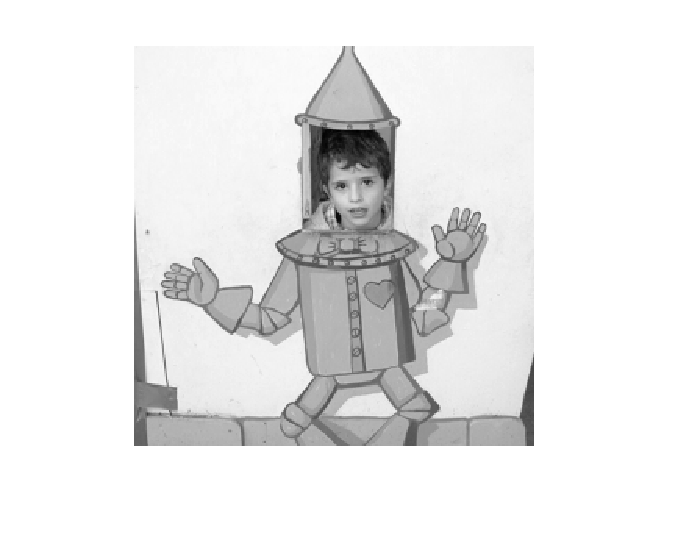}
	}
	\hspace{-1cm}
	\subfigure[Cameraman original ]{
		\includegraphics[width=4.3cm]{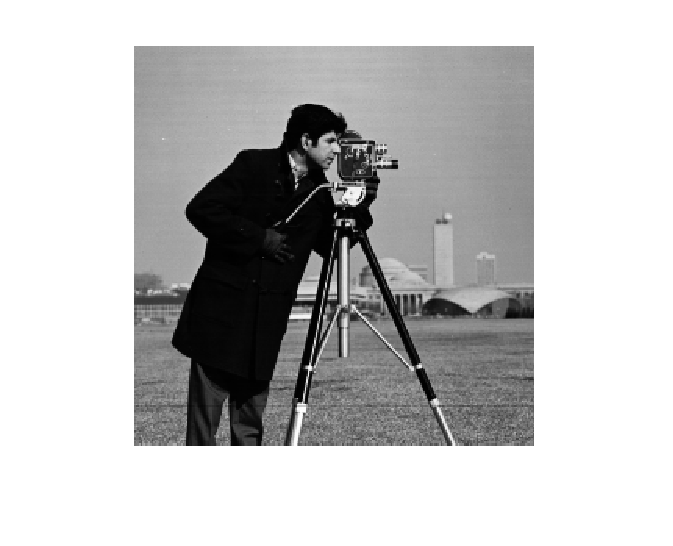}
	}
	\hspace{-1cm}
	\subfigure[Peppers original ]{
		\includegraphics[width=4.3cm]{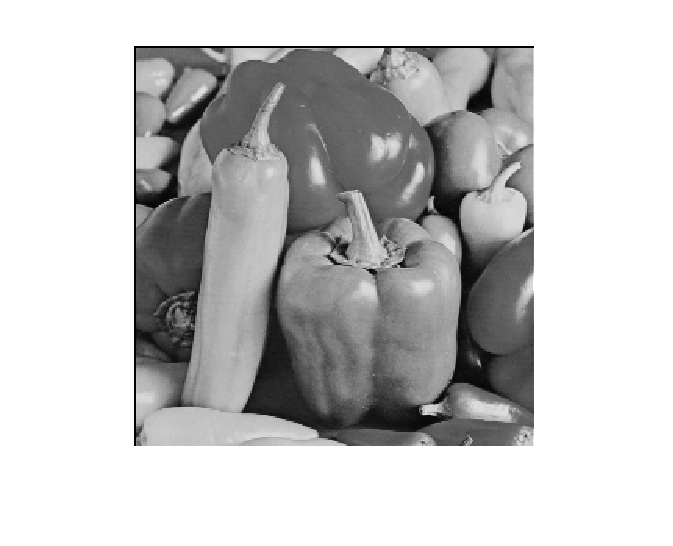}
	}
	\\
	\subfigure[Boy noisy ]{
		\includegraphics[width=4.3cm]{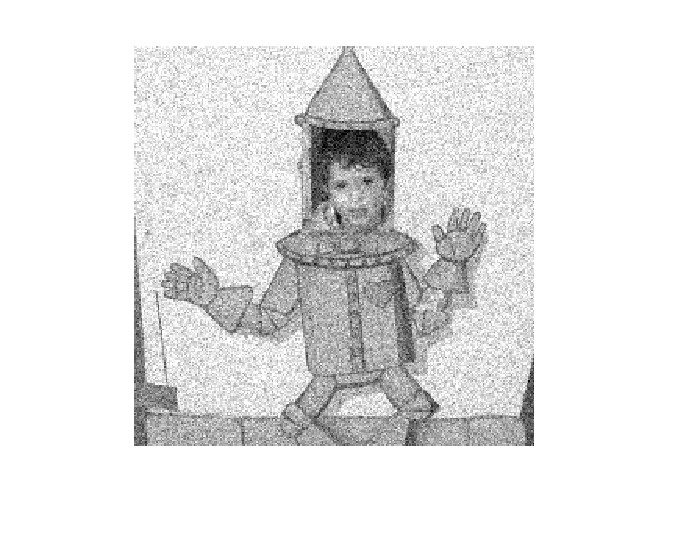}
	}
	\hspace{-1cm}
	\subfigure[Cameraman noisy ]{
		\includegraphics[width=4.3cm]{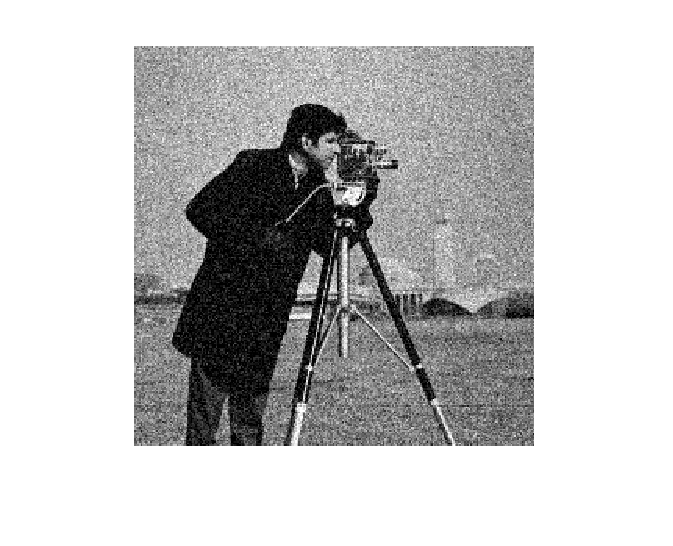}
	}
	\hspace{-1cm}
	\subfigure[Peppers noisy]{
		\includegraphics[width=4.3cm]{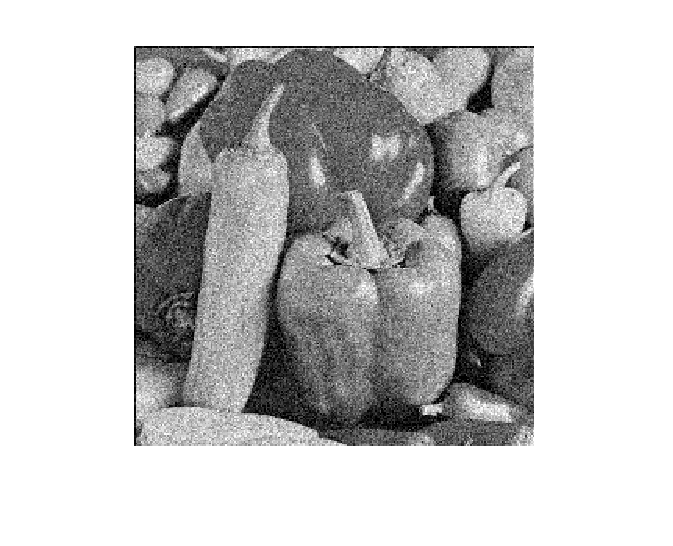}
	}
	\caption{Original images and noise images}
	\label{fig3}	
\end{figure}

\begin{figure}[htbp!]
	\centering
	
	\subfigure[VsaPG]{
		\includegraphics[width=4.3cm]{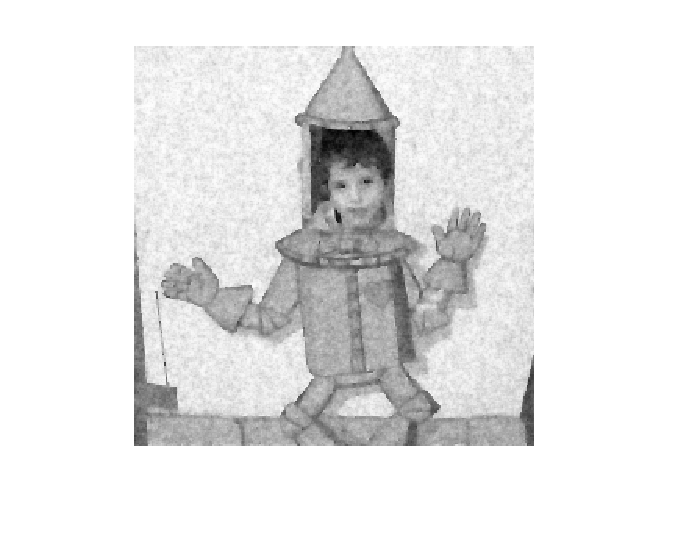}
	}
\hspace{-1cm}
	\subfigure[VsaPG]{
		\includegraphics[width=4.3cm]{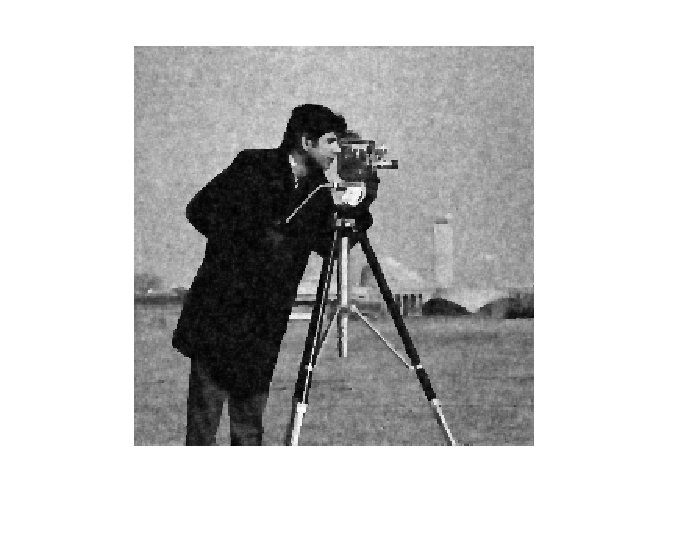}
	}
\hspace{-1cm}
	\subfigure[VsaPG]{
		\includegraphics[width=4.3cm]{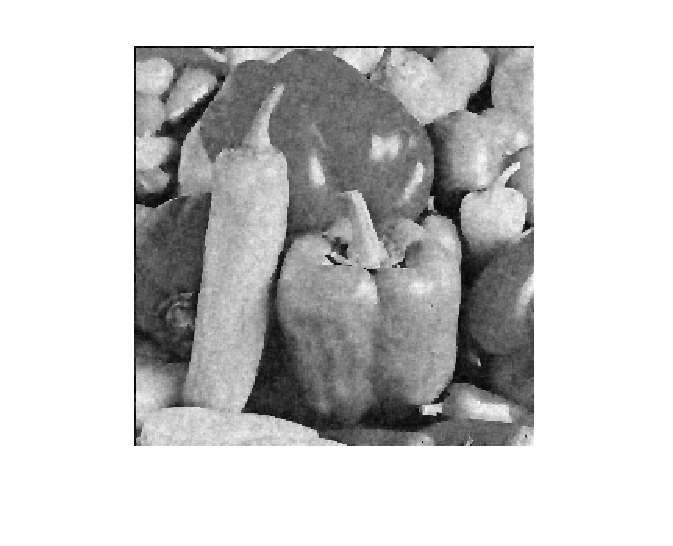}
	}
    \\
    \subfigure[PALM]{
		\includegraphics[width=4.3cm]{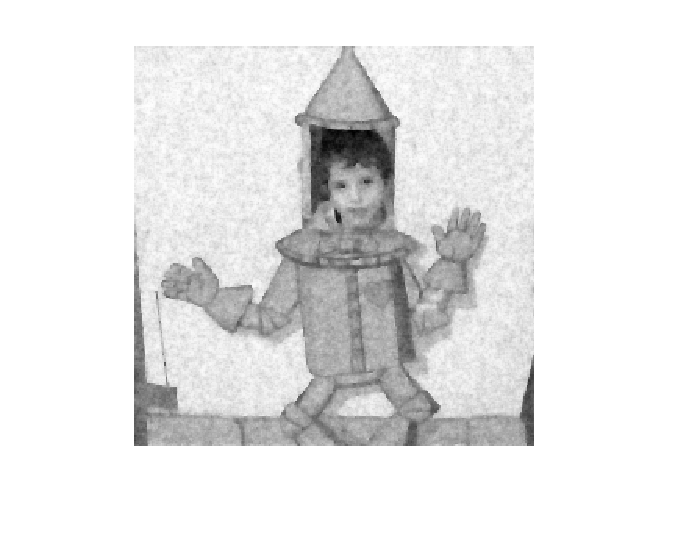}
	}
\hspace{-1cm}
	\subfigure[PALM]{
		\includegraphics[width=4.3cm]{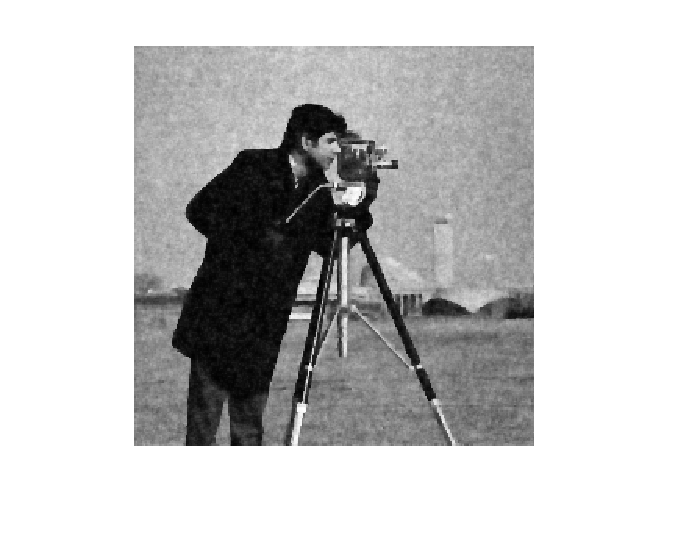}
	}
\hspace{-1cm}
	\subfigure[PALM]{
		\includegraphics[width=4.3cm]{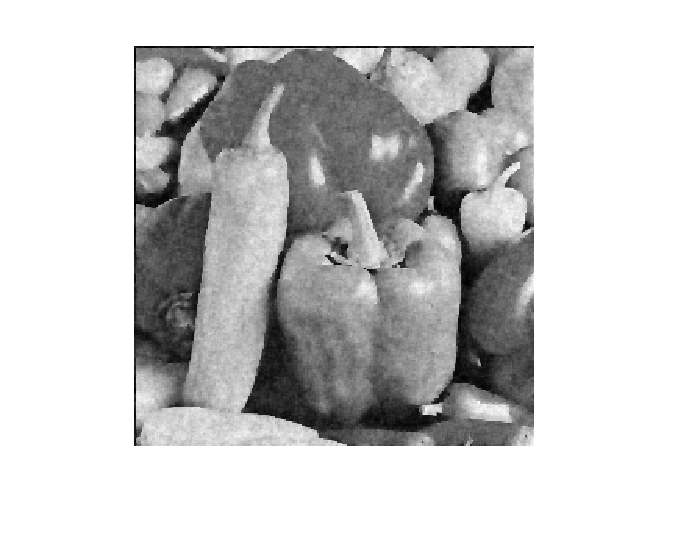}
	}
\\
	\subfigure[GiPALM]{
		\includegraphics[width=4.3cm]{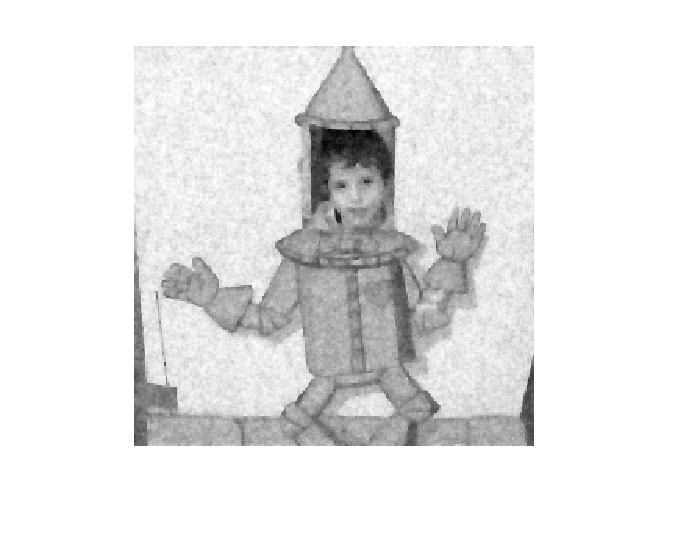}
	}
\hspace{-1cm}
	\subfigure[GiPALM]{
		\includegraphics[width=4.3cm]{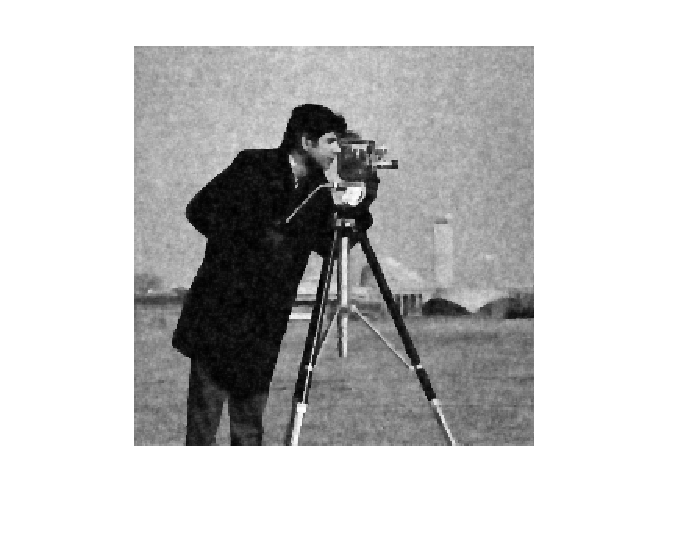}
	}
\hspace{-1cm}
	\subfigure[GiPALM]{
		\includegraphics[width=4.3cm]{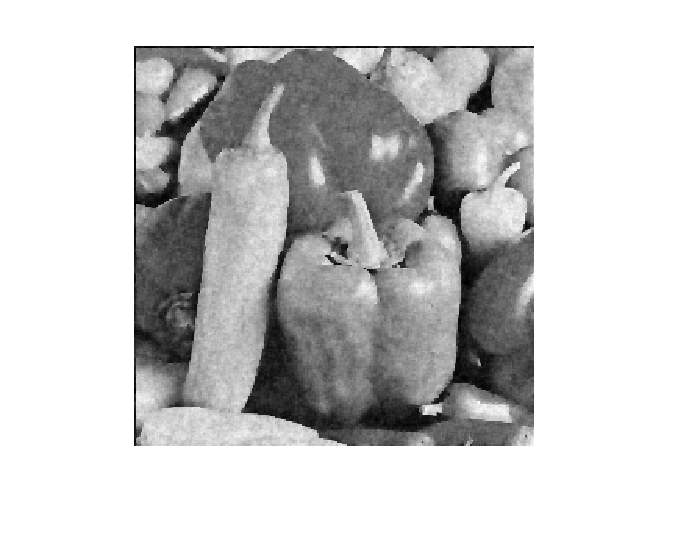}
    }
\\
\subfigure[NiPALM]{
		\includegraphics[width=4.3cm]{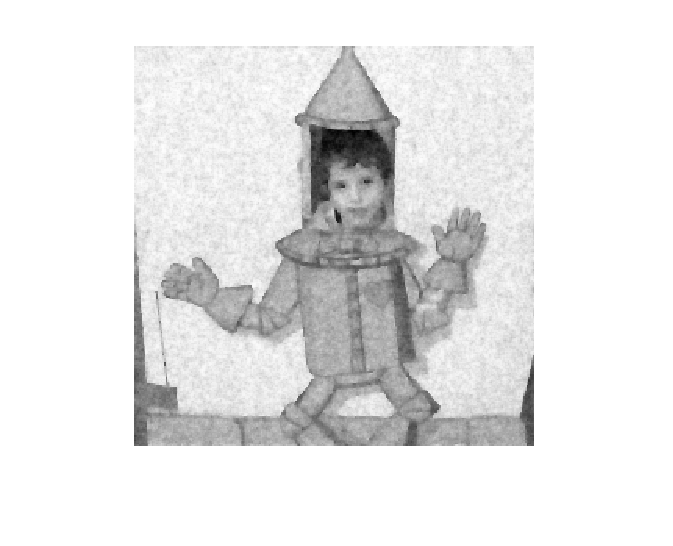}
	}
\hspace{-1cm}
	\subfigure[NiPALM]{
		\includegraphics[width=4.3cm]{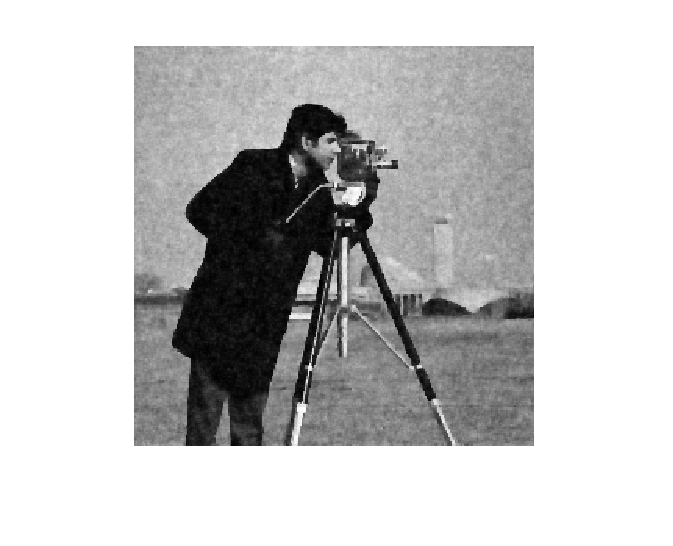}
	}
\hspace{-1cm}
	\subfigure[NiPALM]{
		\includegraphics[width=4.3cm]{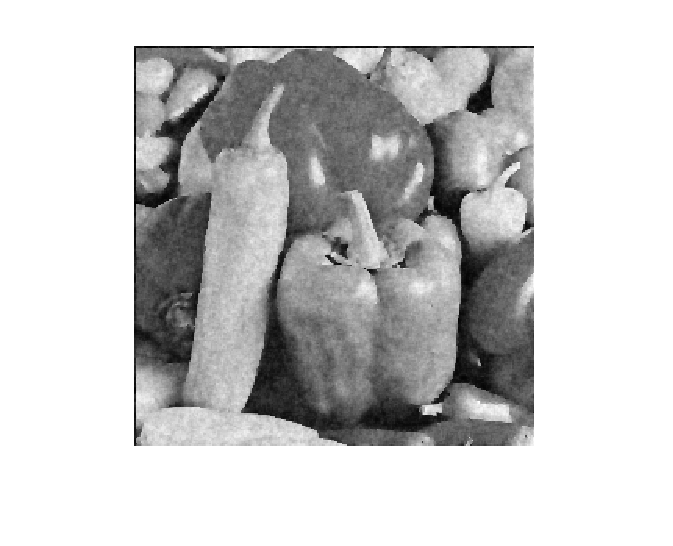}
	}
	\caption{Different Algorithms regarding the numerical effects on different images}
	\label{fig4}	
\end{figure}

\begin{table}[htbp]
  \centering
  \caption{Comparison of different Algorithms for different images}
  \renewcommand{\arraystretch}{1.5}
  \resizebox{\textwidth}{!}{
  \begin{tabular}{cccccccccc}
    \toprule
    \multirow{2}*{Algorithm} & \multicolumn{3}{c}{$Boy$} & \multicolumn{3}{c}{$Cameraman$} & \multicolumn{3}{c}{$Peppers$} \\ \cline{2-10}
                 & Iter & Time    &SNR        & Iter & Time     &SNR        & Iter  & Time     &SNR         \\
    \midrule
    VsaPG       & \textbf{149}  & \textbf{15.2344} &\textbf{38.4914}    & \textbf{176}  & \textbf{18.0000}  &\textbf{33.8557}    & \textbf{219}   & \textbf{22.9531}  &\textbf{36.0595}   \\
    PALM         & 266  & 26.3594 &38.0043    & 276  & 27.4844  &33.5723    & 446   & 46.3125  &34.9240   \\
    GiPALM       & 221  & 22.1092 &38.0039    & 230  & 23.3906  &33.5746    & 371   & 38.6094  &34.9306   \\
    NiPALM       & 248  & 26.2656 &38.2167    & 262  & 27.2188  &33.6710    & 401   & 42.6875  &35.2014   \\
    \bottomrule
  \end{tabular}
  \label{tab6}
  }

\end{table}

\begin{figure}[htbp!]
    \centering
    \begin{minipage}[t]{0.3\textwidth}
        \centering
        \includegraphics[width=\textwidth]{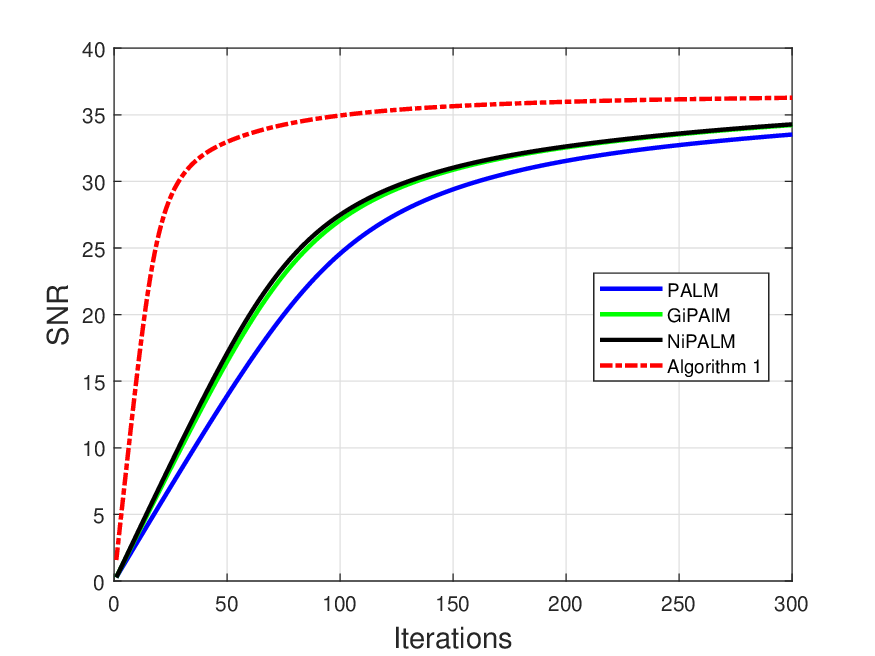}

        \caption*{(a) Peppers}
    \end{minipage}
    \begin{minipage}[t]{0.3\textwidth}
        \centering
        \includegraphics[width=\textwidth]{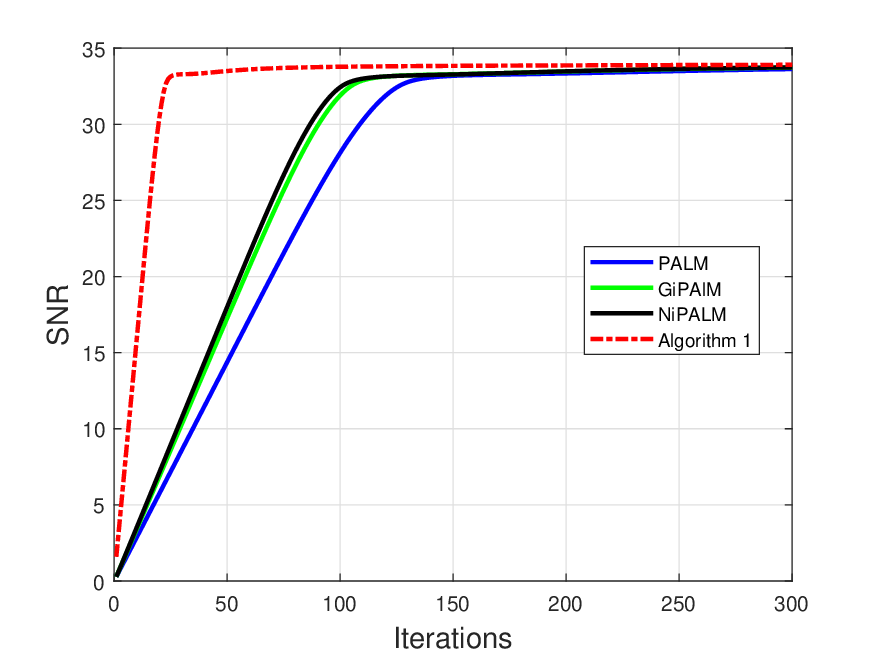}
        \caption*{(b) Cameraman}
    \end{minipage}
    \begin{minipage}[t]{0.3\textwidth}
        \centering
        \includegraphics[width=\textwidth]{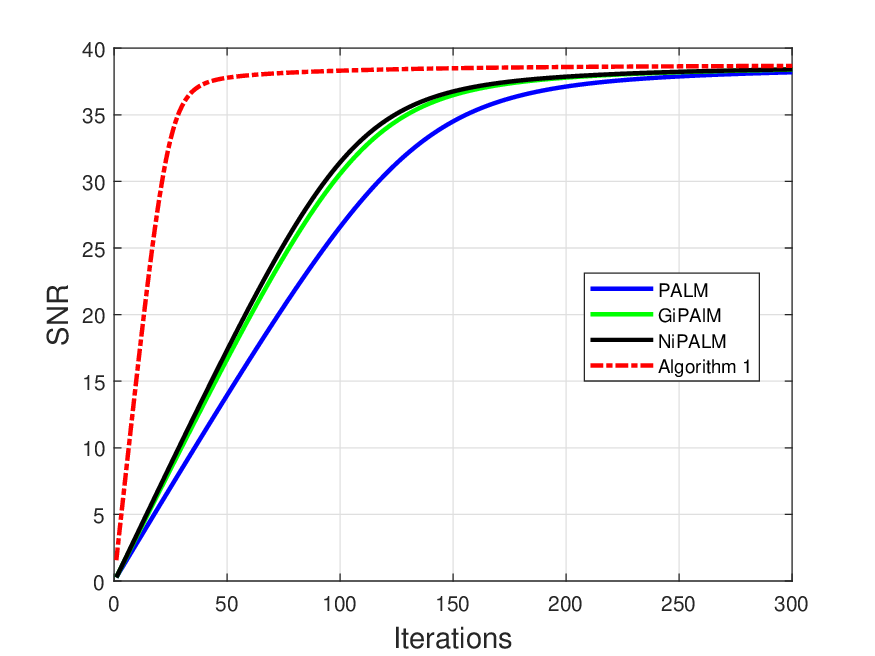}
        \caption*{(c) Boy}
    \end{minipage}
    \caption{SNR values of different test problems}
    \label{fig5}
\end{figure}

\section{Conclusion}

We have proposed a variable smoothing alternating proximal gradient algorithm for solving \eqref{eq:L1E1}, which integrates first-order methods with variable smoothing techniques and allows flexible choices of step sizes and smoothing parameters. Under suitable assumptions, an iteration complexity of $\mathcal{O}(\varepsilon^{-3})$ has been established for obtaining an $\varepsilon$-approximate solution. Numerical experiments on sparse signal recovery and image denoising problems have shown that the proposed algorithm outperforms existing methods.

\end{document}